\documentclass{amsart}

\usepackage{cleveref}

\theoremstyle{plain}
\newtheorem{theorem}{Theorem}
\numberwithin{theorem}{section}
\newtheorem{lemma}[theorem]{Lemma}
\newtheorem{proposition}[theorem]{Proposition}
\newtheorem{corollary}[theorem]{Corollary}

\newtheorem{question}[theorem]{Question}

\theoremstyle{definition}
\newtheorem{definition}[theorem]{Definition}
\newtheorem{remark}[theorem]{Remark}
\newtheorem{example}[theorem]{Example}

\newcommand{\R}{\mathbb R}
\newcommand{\N}{\mathbb N}

\newcommand{\M}{\mathcal M}

\newcommand{\fiber}{\mathcal F}
\newcommand{\baues}{\mathcal B}
\newcommand{\level}[2]{#1^{(#2)}}
\newcommand{\slevel}[3]{#1^{(#2)#3}}
\newcommand{\cyclic}[2]{\mathbf{C}(#1,#2)}
\newcommand{\polygon}[1]{\mathbf{P}_#1}
\newcommand{\Gale}[1]{\mathbf{G}_#1}

\DeclareMathOperator{\conv}{conv}

\DeclareMathOperator{\GKZ}{GKZ}
\DeclareMathOperator{\vol}{vol}
\DeclareMathOperator{\vertices}{vertices}

\DeclareMathOperator{\coh}{coh}
\DeclareMathOperator{\espan}{span}
\DeclareMathOperator{\up}{uh}
\newcommand{\UP}{\mathcal U}
\newcommand{\DOWN}{\mathcal D}

\usepackage[colorinlistoftodos]{todonotes}
\usepackage{enumerate}
\usepackage{tikz-cd}

\title{Hypersimplicial subdivisions}

\author{Jorge Alberto Olarte \and Francisco Santos }

\address[J.~A.~Olarte]
{
Institut f\"ur Mathematik, Freie Universit\"at Berlin, Germany
}
\email{olarte@zedat.fu-berlin.de}

\address[F.~Santos]
{
Dept.~of Mathematics, Statistics and Comp.~Sci., Univ.~of Cantabria, Spain
}
\email{francisco.santos@unican.es}

\thanks{The authors were supported by the Einstein Foundation Berlin under grant EVF-2015-230.
 Work of F. Santos is also supported by project MTM2017-83750-P of the Spanish Ministry of Science (AEI/FEDER, UE)}

\begin{document}

\begin{abstract}
Let $\pi:\R^n \to \R^d$ be any linear projection, let $A$ be the image of the standard basis. 
Motivated by Postnikov's study of postitive Grassmannians via plabic graphs and Galashin's connection of plabic graphs to slices of zonotopal tilings of 3-dimensional cyclic zonotopes, we study the poset of subdivisions induced by the restriction of $\pi$ to 
 the $k$-th hypersimplex, for $k=1,\dots,n-1$.
We show that:

\begin{itemize}
\item For arbitrary $A$ and for $k\le d+1$, the corresponding fiber polytope $\mathcal F^{(k)}(A)$ is normally isomorphic to the Minkowski sum of the secondary polytopes of all subsets of $A$ of size $\max\{d+2,n-k+1\}$.

\item When $A=\polygon{n}$ is the vertex set of an $n$-gon, we answer the Baues question in the positive:
the inclusion of the poset of $\pi$-coherent subdivisions into the poset of all $\pi$-induced subdivisions is a homotopy equivalence.

\item When $A=\cyclic{d}{n}$ is the vertex set of a cyclic $d$-polytope with $d$ odd and any $n \ge d+3$, there are non-lifting (and even more so, non-separated) $\pi$-induced subdivisions for $k=2$. 
\end{itemize}
\end{abstract}

\maketitle
\setcounter{tocdepth}{1}
\tableofcontents

\section{Introduction}
The main object of study in this paper are hypersimplicial subdivisions, defined as follows.
Let $A$ be a set of $n$ points affinely spanning $\R^d$. Let $\Delta_n$ be the standard $(n-1)$-dimensional simplex in $\R^n$. 
Consider the linear projection $\pi: \R^n \to\R^d$ sending the vertices of $\Delta^n$ to the points in $A$. (We implicitly consider the points in $A$ labelled by $[n]$, so that $\pi$ sends $e_i$ to the point labelled by $i$). Let $\level{\Delta_n}k := k\Delta_n \cap [0,1]^n$ be the standard hypersimplex and $\level Ak$ the image of the vertices of $\level{\Delta_n}k$ under $\pi$ (so that points in $\level Ak$ are labelled by $k$-subsets of $[ n]$). A \emph{hypersimplicial subdivision} of $\level Ak$ is a polyhedral subdivision of $\conv(\level Ak)$ such that every face of the subdivision is the image of a face of $\level{\Delta_n}k$ under $\pi$. Put differently, we call hypersimplicial subdivisions the $\pi$-induced subdivisions of the projection $\pi: \level{\Delta_n}k \to \conv(\level Ak)$, as introduced in \cite{BilStu,BKS} (see also \cite{Reiner,triangbook}). See more details in \Cref{sec:preliminaries}.

One reason to study such subdivisions comes from the case where $A\subset \R^2$ are the vertices of a convex polygon. Galashin \cite{Galashin} shows that in this case fine hypersimplicial subdivisions, which we call \emph{hypertriangulations}, are in bijection with maximal collections of chord-separated $k$-sets. These, in turn, correspond to reduced plabic graphs, \cite{OhPosSpe} which are a fundamental tool in the study of the positive Grassmannian \cite{Postnikov2006,Postnikov}. 

More generally, it is of interest the case where $A$ are the vertices of a cyclic polytope $\cyclic nd \subset \R^d$.
(The $n$-gon is the case $d=2$).
In \cite[Problem 10.3]{Postnikov} Postnikov asks the \emph{generalized Baues problem} for this scenario; that is, he asks whether the poset of hypersimplicial subdivisions of $\level{\cyclic nd}k$ has the homotopy type of a $(n-d-2)$-sphere. For $k=1$ this was shown to have a positive answer by Rambau and Santos \cite{RaSa}. 
For $d=2$, Balitskiy and Wellman show the poset to be simply connected and again ask the Baues question for it (\cite[Theorem 6.4 and Question 6.1]{BalWel}). We here give the answer to this:

\begin{theorem}
\label{thm:baues-intro}
Let $\polygon n$ be the vertices of any convex $n$-gon. The poset of hypersimplicial subdivisions $\baues(\level{\Delta_n}k \to \level {\polygon n}k)$ retracts onto the poset of \emph{coherent} hypersimplicial subdivisions. In particular, it has the homotopy type of an $(n-4)$-sphere.
\end{theorem}

\cite[Problem 10.3]{Postnikov} also asks for which values of the parameters can all hypersimplicial subdivisions of $\level{\cyclic nd}k$ be lifted to zonotopal tilings of the cyclic zonotope. This was already known to be false for $d=1$ \cite[Example 10.4]{Postnikov} and we generalize the counterexamples to every odd dimension:

\begin{theorem}
\label{thm:separated-intro}
Consider the cyclic polytope $\cyclic nd\subset \R^d$ for odd $d$ and $n\ge d+3$. Then, for every $k\in [2,n-2]$ there exist hypersimplicial subdivisions of $\level{\cyclic nd}k$  that do not extend to zonotopal tilings of the cyclic zonotope $Z(\cyclic nd)$.
\end{theorem}

In contrast, Galashin \cite{Galashin} showed that the answer to Postnikov's question is positive in dimension two for  \emph{hypertriangulations}, a result that was generalized to all hypersimplicial \emph{subdivisions} by 
Balitskiy and Wellman \cite[Lemma 6.3]{BalWel}.

The poset of coherent hypersimplicial subdivisions of any $A$ is isomorphic to the face poset of a polytope, a particular case of a fiber polytope. When $k=1$ this is just the secondary polytope of $A$, so for $k>1$ we call it the \emph{$k$-th hypersecondary polytope of $A$}. We study hypersecondary polytopes for any $A\subset \R^d$ and $k\le d+1$. Specifically, we show that these polytopes are normally equivalent to the Minkowski sum of certain faces of the secondary polytope of $A$.
By symmetry, an analogue statement holds  for $n-d-1\le k<n$.

\begin{theorem}
\label{thm:fiber-intro}
Let $A\subseteq \R^d$ be a configuration of size $n$ and $k\in [d+1]$. Let $s = \max (n-k+1, d+2)$.  The hypersecondary polytope $\fiber^{(k)}(A)$ is normally equivalent to the Minkowski sum of the secondary polytopes of all subsets of $A$ of size $s$.
\end{theorem}

The paper is organized as follows: \Cref{sec:preliminaries} introduces notation and basic background on induced subdivisions in general, and hypersimplicial subdivisions in particular. In \Cref{sec:fiber} we look at coherent hypersimplicial subdivisions and hypersecondary polytopes as Minkowski sums and prove \Cref{thm:fiber-intro}, among other results.
In \Cref{sec:separation} we study the connection of hypersimplicial subdivisions with zonotopal tilings. In particular, we extend to tiles of positive dimension the concept of $A$-\emph{separated sets} introduced in \cite{GaPo}.
With this machinery we show that if all hypertriangulations of $A$ are separated then all hypersubdivisions are separated too (\Cref{coro:fine_is_enough}).
In \Cref{sec:non-separated} and \Cref{sec:baues} we prove \Cref{thm:separated-intro} and \Cref{thm:baues-intro} respectively. Finally, we briefly discuss the enumeration of hypersimplicial subdivisions of $\level{\polygon n}2$ in \Cref{sec:hypercatalan}. 

\subsection*{Acknowledgements}
We thank Alexander Postnikov for inspiring us to work on this and  Alexey Balitskiy, Pavel Galashin and Julian Wellman for comments on a first version of this paper.

\section{Preliminaries and notation}
\label{sec:preliminaries}

\subsection{Fiber polytopes}

We here briefly recall the main concepts and results on fiber polytopes. See \cite{BilStu} or \cite{Reiner} for more details.

Let $\pi: \R^n \to \R^d$ be a linear projection map. 
Let $Q\subset \R^n$ be a polytope and let $A=\pi(\vertices(Q)$). 
A $\pi$-induced subdivision of $A$ is a polyhedral subdivision $S$ (in the sense of, for example, \cite{triangbook}), such that every face of $S$ is the image under $\pi$ of a face $F$ of $Q$. 

Given a vector $w\in (\R^n)^*$ the face $Q^w$ of $Q$ selected by $w$ is the convex hull of all vertices of $Q$ which minimize $w$. A \emph{$\pi$}-coherent subdivison is a $\pi$-induced subdivision in which the faces of $Q$ are chosen ``coherently'' via a vector $w\in (\R^n)^*$. More precisely, we define the \emph{$\pi$-coherent subdivision of $A$ given by $w$} to be
\[
S(Q \stackrel{\pi}{\to} A,w)
:=
\left\{ \pi(F) : \exists \tilde w\in (\R^n)^* \text{ s.t. } \tilde w|_{\ker(\pi)} = w|_{\ker(\pi)} \text{ and } Q^{\tilde w} = F\right\}.
\]

The \emph{fiber fan} of the projection $Q \stackrel{\pi}{\to} A$ is the stratification of $(\R^n)^*$ according to what $\pi$-coherent subdivision is produced. It is a polyhedral fan with linearity space equal to 
\[
\{w \in (\R^n)^* : \ker(\pi) \subset \ker(w)\} + \{w \in (\R^n)^* : w|_Q=\text{constant}\}.
\]
As we will see below, it is the normal fan of a certain polytope $\fiber(Q \stackrel{\pi}{\to} A)$ of dimension $\dim(Q) - \dim(A)$.

To define $\fiber(Q \stackrel{\pi}{\to} A)$, we look at fine $\pi$-induced subdivisions.
A $\pi$-induced subdivision $S$ is \emph{fine} if $\dim(F)=\dim(\pi(F))$ for each of the faces $F\le Q$ whose images are cells in $S$
Put differently, a fine $\pi$-induced subdivision is the image of a subcomplex of $Q$ that is a section of $\pi:Q\to \conv(A)$. To each fine $\pi$-induced subdivision $S$ we associate 
the following point:
\[
\GKZ(S):= \sum_{F\le Q \atop \pi(F)\in S} \frac{\vol(\pi(F))}{\vol(A)} {\bf c}(F) \in \R^n, 
\]
where ${\bf c}(F)$ denotes the centroid of $F$.

\begin{definition}
The \emph{fiber polytope} of the projection $\pi: Q \to \conv(A)$ is the convex hull of the vectors $\GKZ(S)$ for all fine $\pi$-induced subdivisions. We denote it $\fiber(Q\to A)$.
\end{definition}

The main property of the fiber polytope is the following result of Billera and Sturmfels. In fact, for the purposes of this paper this theorem can be taken as a \emph{definition} of the fiber polytope, since our results are mostly not about the polytope but about its normal fan (see, eg \Cref{sec:fiber}).

\begin{theorem}[Billera and Sturmfels~\cite{BilStu}]
$\fiber(Q\to A)$ is a polytope of dimension $\dim(Q) - \dim(A)$ whose normal fan equals the fiber fan.
\end{theorem}

In particular, the face lattice of $\fiber(Q\to A)$ is isomorphic to the poset of $\pi$-coherent subdivisions ordered by refinement. For example, vertices of $\fiber(Q\to A)$ correspond bijectively to fine $\pi$-coherent subdivisions.

Two cases of this construction are of particular importance. Let $A=\{a_1,\dots,a_n\}\subset \R^d$ be a configuration of $n$ points. Then:
\begin{enumerate}
\item If we let $\pi: \Delta_n \to \conv(A)$ be the affine map $e_i\mapsto a_i$
bijecting vertices of $\Delta_n$ to $A$, then all the polyhedral subdivisions of $A$ are $\pi$-induced, and the coherent ones are usually called \emph{regular} subdivisions of $A$. The corresponding fiber polytope is the \emph{secondary polytope} of $A$ and we denote it $\level{\fiber}{1}(A)$ (in the next sections we define $\level{\fiber}{k}(A)$ for other values of $k$).

\item Let
\[
Z(A) = \sum_{i} \conv\{0,(a_i,1)\} \subset\R^{d+1}
\]
be the zonotope generated by the \emph{vector} configuration $A\times\{1\}\subset \R^{d+1}$.
The $\pi$ in the previous case extends to a linear map
$\pi: [0,1]^n \to Z(A)$ still sending $e_i\mapsto a_i$. Then the $\pi$-induced subdivisions are precisely the \emph{zonotopal tilings} of $Z(A)$. The corresponding fiber polytope is the \emph{fiber zonotope} of $Z(A)$ (or of $A)$ and we denote it $\fiber^Z(A)$.
\end{enumerate}

\subsection{The Baues problem}

The poset of all $\pi$-induced subdivisions (excluding the trivial subdivision for technical reasons) is called the \emph{Baues poset} of the projection and we denote it $\baues(Q\to A)$. The subposet of $\pi$-coherent subdivisions is denoted $\baues_{\coh}(Q\to A)$. The \emph{Baues problem} is, loosely speaking, the question of how similar are $\baues(Q\to A)$ and $\baues_{\coh}(Q\to A)$, formalized as follows: 

To every poset $\mathcal P$ one can associate a simplicial complex called the \emph{order complex of $\mathcal P$} by using the elements of $\mathcal P$ as elements and chains in the poset as simplices. In particular, one can speak of the \emph{homotopy type} of $\mathcal P$ meaning that of its order complex. Similarly, an order preserving map of posets
\[
f: \mathcal P_1 \to \mathcal P_2
\]
induces a simplicial map between the corresponding order complexes, and one can speak of the homotopy type of $f$.

The prototypical example is the following: if $\mathcal P$ is the face poset of a polyhedral complex $\mathcal C$, then the order complex of $\mathcal P$ is (isomorphic to) the barycentric subdivision of $\mathcal C$. In particular,
since $\baues_{\coh}(Q\to A)$ is the face poset of the polytope $\fiber(Q\to A)$, 
it  is homotopy equivalent (in fact, homeomorphic) to a sphere of dimension $\dim(Q) - \dim(A) -1$.

\begin{question}[Baues Problem]
Under what conditions is the inclusion $\baues_{\coh}(Q\to A) \hookrightarrow \baues(Q\to A)$ a homotopy equivalence?
\end{question}

See \cite{Reiner} for a (not-so-recent) survey about this question, and \cite{Santos2006,Liu2017} for examples where the answer is no and having $Q$ a simplex and a cube, respectively. 

\subsection{Cyclic polytopes}
Cyclic polytopes are a family of polytopes of particular interest for this manuscript and are defined as follows. The trigonometric moment curve (also known as the Carath\'eodory curve), is parametrized by
\[
\phi_d: \enspace t \to (\sin(t),\cos(t),\sin(2t),\cos(2t),\dots)\in \R^d.
\]
Let $t_1,\dots,t_n$ be $n$ cyclically equidistant numbers in $[0,2\pi)$, for example, $t_i = \frac{2\pi(i-1)}{n}$. 
The \emph{cyclic polytope} $\cyclic nd$ is the convex hull of $\phi(t_1),\dots,\phi(t_n)$. 

The combinatorics of the cyclic polytope can be nicely described in terms of the circuits of the corresponding oriented matroid. Namely, all circuits are of the form $(\{a_1,a_3,\dots\}, \{a_2,a_4,\dots\})$ such that $a_1<a_2<\dots<a_{d+2}$ and their opposites (giving the label $i$ to the vertex $\phi(t_i)$). 

Cyclic polytopes can also be defined by using the polynomial moment curve $t\to (t,t^2,\dots,t^d)$ instead of the trigonometric moment curve and the combinatorial type remains the same. However, the coherence of subdivisions and hence fiber polytopes depend also on the embedding (see \Cref{ex:hypersecondary}). When using the trigonometric moment curve in even dimension the cyclic polytope has more symmetry. That is, it is invariant under the cyclic group action on the vertices. When $d=2$ the cyclic polytope $\cyclic n2$ is a is a regular polygon and we abbreviate it by $\polygon n$.

The Baues problem is known to have positive answer for cyclic polytopes in the following two cases:

\begin{theorem}[\cite{RaSa,SturmfelsZiegler1993}]
\label{thm:RaSaStZi}
Let $n >d \in \N$. Then, the following two cases of the Baues question have a positive answer:
\begin{itemize}
\item When $Q=\Delta_n$ and $A=\cyclic{n}{d}$ is the \emph{cyclic polytope} of dimension $d$ with $n$ vertices \cite{RaSa}.
\item When $Q=[0,1]^n$ and $A=Z(\cyclic{n}{d})$ is the \emph{cyclic zonotope} of dimension $d+1$ with $n$ generators \cite{SturmfelsZiegler1993}.
\end{itemize}
\end{theorem}

\subsection{Hypersecondary polytopes.}
Let $A=\{a_1,\dots, a_n\}\in\R^d$ be a point configuration. For each $k=1,\dots,n-1$ we consider the following $k$-th deleted (Minkowski) sum of $A$ with itself, which we denote $A^{(k)}$:
\[
A^{(k)} := \left \{a_{i_1} + \dots + a_{i_k} \in \R^d : \{i_1,\dots, i_k\} \in \binom{[n]}{k}\right\}.
\]

The $k$-th deleted sum of the standard $(n-1)$-simplex $\Delta_{n} :=\conv(e_1,\dots,e_n)$ equals the $k$-th hypersimplex of dimension $n-1$:
\[
\level{\Delta_n}k := \conv\left\{\sum_ {i\in B} e_i : B\in\binom{[n]}{k}\right\} = [0,1]^n \cap \left\{x:\sum_{i=1}^n x_i = k\right\}.
\]
(Observe that the notation $\level{\Delta_n}k$ here is an abbreviation of $\conv(\vertices(\Delta_n)^{(k)}))$.

As mentioned above, the projection $\R^n\to \R^d \times \{1\}$ that sends the vertices of $\Delta_n$ to $A$ extends to a linear map $\R^{n}\to \R^{d+1}$ that sends the unit cube $[0,1]^n$ to the zonotope 
$Z(A)$. In turn, this linear map restricts to an affine map sending each $\level{\Delta_n}k\subset \R^n$ to $A^{(k)}\subset \R^d \times \{k\}$.  We use the same letter $\pi$ for all these projections.

\begin{definition}
The $\pi$-induced subdivisions of the projection $\pi: \level{\Delta_n}k \to A^{(k)}$ are called \emph{hypersimplicial subdivisions of level $k$} of $A$, or just  \emph{hypersimplicial subdivisions} of $A^{(k)}$. Fine hypersimplicial subdivisions are called \emph{hypertriangulations}. We denote $\level{\baues}k(A)$  and $\level\fiber{k}(A)$ the corresponding Baues poset and fiber polytope, and call the latter the  \emph{$k$-th hypersecondary polytope} of $A$. We denote $\level{\baues_{\coh}}k(A)$ for the coherent subdivisions in $\level{\baues}{k}(A)$.
\end{definition}

\begin{remark}
The Baues poset $\level{\baues}k(A)$ only depends on the oriented matroid of $A$ while $\level{\baues_{\coh}}k(A)$ does depend on the embedding of the oriented matroid. 
\end{remark}

\subsection{Lifting subdivisions}

By construction, the intersection of any zonotopal tiling of $Z(A)$ with the hyperplane $\sum x_i = k$ is a hypersimplicial subdivision of $A^{(k)}$.  But the converse is in general not true. Not every hypersimplicial subdivision of $A^{(k)}$ ``extends'' to a zonotopal tiling of $Z(A)$. Following \cite{OMbook,Postnikov,OMtri} the ones that extend are called \emph{lifting} hypersimplicial subdivisions. The following are examples of them:

\begin{itemize}
\item For a cyclic polytope $\cyclic nd$, all triangulations in the standard sense (that is, all hypertriangulations of $\level{\cyclic nd}1$) are lifting~\cite{RaSa}. The same is not known for non-simplicial subdivisions.

\item For arbitrary $k$ and a convex $n$-gon $\polygon n$, all hypertriangulations of $\level{\polygon n}k$ are lifting~\cite{Galashin}. The same result for all hypersimplicial subdivisions has recently been provedin \cite{BalWel}.
\end{itemize}

Non-lifting triangulations of $A^{(1)}$ are not known in dimension two but easy to construct in dimension three or higher. For example, if a subdivision $S$ of $A$ has the property that its restriction to some subset $B$ of $A$ cannot be extended to a subdivision of $B$, then $S$ is non-lifting. Such subdivisions (and triangulations) exist when $A$ is the vertex set of a triangular prism together with any point in the interior of it, the vertex set of a $4$-cube, or the vertex set of $\Delta_4\times \Delta_4$, among other cases (see, e.g.,  \cite[Chapter~5]{OMtri}, or 
\cite[Proof (10) in Sect.~7.1.2, ]{triangbook}).

To better understand lifting subdivisions, let us look at zonotopal tilings of $Z(A)$. 
We denote $\baues^Z (A)$, $\baues_{\coh}^Z (A)$ and $\fiber^Z (A)$ for the poset of zonotopal tilings, its subposet of coherent tilings and the secondary zonotope of $Z(A)$ respectively.
We call any subset of $[n]$ a \emph{point}, since it represents an element of the point configuration $\sum_{i\in [n]} \{0, a_i\}$.
A \emph{tile} is a poset interval $[X,Y]$ of the boolean poset $2^{[n]}$, where $X\subseteq Y$. To be precise, $[X,Y] := \{I \subseteq [n] \mid X\subseteq I \subseteq Y\}$. Geometrically, we think of $[X,Y]$ as the zonotope $X + Z(Y\setminus X)$,
but we prefer the combinatorial notation where the tile is described as the set of vertices of $[0,1]^n$ of which it is the projection. 

Every tile is a cell in a coherent zonotopal tiling of $Z(A)$, by letting $w(j)$ be $-1$, $0$ or $1$ depending on whether $j$ is in $X$, $Y\setminus X$, or none of them. Indeed, this $w$ gives value at least $-|X|$ to every point in $Z(A)$, with equality if and only if the point belongs to $[X,Y]$.

Turning our attention to hypersimplices, observe that every face of the hypersimplex $\level {\Delta_n} k$ is the intersection of a face of $[0,1]^n$ with the hyperplane $\left\{x:\sum_{i=1}^n x_i = k\right\}$. Therefore we can denote the projection under $\pi$ of any face of $\level {\Delta_n} k$ by
\[
\level{[X,Y]}{k} := [X,Y]\cap \left(\R^d\times\{k\}\right) = \{B \mid X\subseteq B \subseteq Y \quad |B| = k\}.
\]
By definition, a subdivision of $\level Ak$ is hypersimiplicial if and only if all of its cells are of the form $\level{[X,Y]}k$. A hypersimplicial subdivision is fine if for every cell $[X,Y]^k$ we have that $Y/X$ is an affine basis in $A$. This spells out the following relation with zonotopal tilings:

\begin{proposition}
\label{prop:intersecting}
For every configuration $A$ of $n$ points and every $k\in [n-1]$:
\begin{enumerate}
\item Intersection of zonotopal tilings with the hyperplane at level $k$ induces an order-preserving map
\[
\level{r}{k}: \baues^Z(A) \to \level\baues{k}(A).
\]
\item The normal fan of $\fiber^Z(A)$ refines the normal fan of $\level\fiber{k}(A)$.
\end{enumerate}
\end{proposition}
\begin{proof}
For the first claim, notice that the intersection of a zonotopal tiling $S = \{[X_i,Y_i] \mid i \in I\}$ with the hyperplane $\R^d\times k$ gives the subdivision
\[
\level rk (S) := \left\{\level{[X_i,Y_i]}k \mid i\in I \enspace |X_i|<k<|Y_i|\right\}\cup\left\{X\in {n \choose k}\cap S\right\}
\] 
of $\level Ak$, which clearly is hypersimplicial. We denote $\level rk (S)$ as $\level Sk$ for simplicity. 
The second claim follows from the fact that $\level{S(Z(A),w)}{k} = S(\level Ak, w)$ for every $w\in (\R^n)^*$. 
\end{proof}
We say that a tile $[X,Y]$ \emph{covers level} $k$, if $|X|<k<|Y|$. In other words, $[X,Y]$ covers level $k$ if $\level{[X,Y]}k$ is of positive dimension.

\begin{example}
\label{ex:non_coherent}
Consider the regular hexagon $\polygon 6$. \Cref{fig:not_coherent} shows a hypersimplicial subdivision of $\level {\polygon 6}2$ whose set of facets are the triangles $\level{[\emptyset,123]}2$, $\level{[\emptyset,135]}2$, $\level{[\emptyset,156]}2$, $\level{[\emptyset,345]}2$, $\level{[1,1236]}2$,
$\level{[1,1356]}2$, $\level{[3,1235]}2$, $\level{[3,2345]}2$, $\level{[5,1345]}2$ and $\level{[5,1456]}2$. The colour of the triangle $\level{[X,Y]}2$ is dark gray if $X =\emptyset$ and white if $|X|=1$, which agrees with the colouring of vertices of the corresponding plabic graph (see \cite{Galashin}).

\begin{figure}[h]
	\centering
		\includegraphics[width=0.4\textwidth]{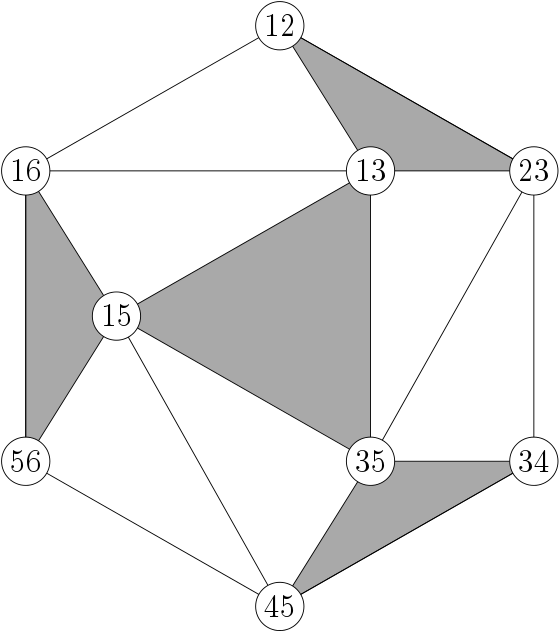}
	\caption{A non-coherent hypersimplicial subdivision of $\level{\polygon 6}2$.}
	\label{fig:not_coherent}
\end{figure}

This subdivision is not coherent. To see this, suppose there is a lifting vector $w\in(\R^*)^6$ whose regular subdivision is this. Then notice that the presence of the edge $\level{[1,136]}2$ implies $w_3+w_6<w_2+w_5$, the presence of the edge $\level{[3,235]}2$ implies $w_2+w_5<w_1+w_4$ and the presence of the edge $\level{[5,145]}2$ implies $w_1+w_4<w_3+w_6$, together forming a contradiction. This contrasts the fact that every subdivision of a convex polygon is regular. 
\end{example}

\subsection{Lifting subdivisions via Gale transforms. The Bohne-Dress Theorem}

As a general reference for the contents of this section we recommend the book \cite{triangbook}, more specifically Chapters 4, 5 and 9.

A \emph{Gale transform} of a point configuration $A=\{a_1,\dots,a_n\}$ is a vector configuration $\Gale{A}=\{a^*_1,\dots, a^*_n\}$ with the property that a vector $(\lambda_1,\dots,\lambda_n)\in \R^n$ is the coefficient vector of an affine dependence in $A$ if and only if it is the vector of values of a linear functional on $\Gale{A}$. The definition implicitly assumes a bijection between $A$ and $\Gale{A}$ given by the labels $1,\dots,n$. 

Gale duality is an involution: the Gale duals of a Gale dual of $A$ are linearly isomorphic to $A$ when considering $A$ as a vector configuration via \emph{homogenization}, by which we mean looking at affine geometry on the \emph{points} $a_1,\dots,a_n$ as linear algebra on the \emph{vectors} $(a_1,1), \dots, (a_n,1)$. In fact, if $A$ and $B$ are Gale duals to one another then their oriented matroids are dual, which implies that their ranks add up to $n$. In our setting where $A$ has affine dimension $d$ and hence rank $d+1$, its Gale duals have rank $n-d-1$.

The normal fan of the secondary polytope $\level{\fiber}{1}(A)$ of $A$ lives naturally in the ambient space of $\Gale{A}$: it equals the common refinement of all the complete fans with rays taken from $\Gale{A}$. 
Put differently, vectors $w \in \espan(\Gale{A})$ 
are in natural bijection to lifting functions $A\to \R$ (where the latter, which forms a linear space isomorphic to $\R^n$, is considered modulo the linear subspace of affine functions restricted to $A$). 
Under this identification, $w_1$ and $w_2$ define the same coherent subdivision of $A$ if and only if they lie in exactly the same family of cones among the finitely many cones spanned by subsets of $B$. 
The precise combinatorial rule to construct the coherent subdivision $S=S(\Delta_n \stackrel{\pi}{\to} A,w)$ of $A$ induced by a $w\in \espan( \Gale{A} )$ is: \emph{a subset $Y\subset [n]$ is a cell in $S$ if and only if $w$ lies in the relative interior of $[n]\backslash Y$.}

This rule can be made purely combinatorial as follows. 
Instead of starting with a vector $w\in \espan(\Gale{A})$, let $\M^*(A)$ be the oriented matroid of $\Gale{A}$ and let $\M'$ be a \emph{single-element extension} of $\M^*(A)$. 
That is, $\M'$ is an oriented matroid of the same rank as $\M$ on the ground set $[n]\cup\{w\}$ and such that $\M'$ restricted to $[n]$ equals $\M^*(A)$. 
Any vector $w\in \espan( \Gale{A})$ induces such an extension, but the definition is more general since $\M'$ needs not be realizable, or it may be realizable but not extend the given realization $\Gale{A}$ of $\M^*(A)$. 
Yet, any such extension $w$ allows to define a subdivision $S(w)$ of $A$ as follows.

\begin{proposition}
\label{prop:lifting}
With the notation above, the following rules define, respectively, a polyhedral subdivision $\level{S}{1}(A,w)$ of $A$ and a zonotopal tiling $\level{S}{Z}(A,w)$ of $Z(A)$:
\begin{enumerate}
\item A  subset $Y\subset [n]$ is a cell in $\level{S}{1}(A,w)$ if and only if $([n]\setminus Y, \{w\})$ is a vector in the oriented matroid $\M'$. 

\item An interval $[X,Y]$ is a tile in $\level{S}{Z}(A,w)$ if and only if $([n]\setminus Y, X \cup \{w\})$ is a vector in the oriented matroid $\M'$. 
\end{enumerate}
\end{proposition}

By construction, $\level{S}{1}(A,w)$ is the slice at height $1$ of $\level{S}{Z}(A,w)$. In fact:

\begin{theorem}[Bohne-Dress Theorem]
\label{thm:BohneDress}
The construction of \Cref{prop:lifting}(2) is a bijection (and a poset isomorphism, with the weak map order on extensions of $\M^*(A)$) between one-element extensions of $\M^*(A)$ and zonotopal tilings of $Z(A)$. In particular, lifting subdivisions of $\level{A}{1}$ are precisely the ones that can be obtained by the construction in \Cref{prop:lifting}(1).
\end{theorem}

\section{Normal fans of hypersecondary polytopes}
\label{sec:fiber}

The goal of this section is to study hypersecondary polytopes, and the relations between them and the secondary zonotope.
Most of such relations say that the normal fan of one of the polytopes refines that of another one. We introduce the following definition to this effect:

\begin{definition}
Let $P, Q \in \R^d$ be two polytopes. We say that $Q$ is a \emph{Minkowski summand} of $P$, and write $Q \le P$, if any of the following equivalent conditions holds:
\begin{enumerate}
\item The normal fan of $P$ refines that of $Q$.
\item $P+Q$ is combinatorially isomorphic to $P$.

\end{enumerate}
If $P$ and $Q$ are Minkowski summands of one another then they are \emph{normally equivalent} and we write $P\cong Q$.
\end{definition}

\begin{remark}
The equivalence of these two conditions follows from the fact that the normal fan of $P+Q$ is the common refinement of the normal fans of $P$ and $Q$. It can be shown $Q\le P$ is also equivalent to the existence of a polytope $Q'$ and an $\varepsilon >0$ such that $P=Q'+\varepsilon Q$, hence the name ``Minkowski summand''.
\end{remark}

Throughout this section we will assume that $A\subseteq \R^d$ is a point configuration that spans affinely $\R^d$. As a first example, it follows from \Cref{prop:intersecting} that:

\begin{proposition}
For every configuration $A\subset \R^d$ of size $n$:
\label{prop:fiber_zonotope}
\begin{enumerate}
\item $\fiber^{(k)}(A) \le \fiber^Z(A)$.

\item Let $k_0=0 < k_1 < \dots < k_p=n$ be a sequence of integers with $k_{i+1}-k_i \le d+1$ for all $i$. Then,
\[
 \fiber^Z(A) \cong \sum_{i=0}^p \fiber^{(k_i)}(A).
\]
\end{enumerate}
\end{proposition}

In particular:

\begin{corollary}
For every configuration $A\subset \R^d$ of size $n$,
\label{coro:fiber_zonotope}
\begin{enumerate}
\item If $n \le 2d+2$ then
\[
\fiber^Z(A) \cong \fiber^{(k)}(A) , \qquad \forall k\in[n-d-1,d+1].
\]
\item If $n\ge 2d+2$ then
\[
\sum_{k=d+1}^{n-d-1}  \fiber^{(k)}(A) \cong \fiber^Z(A).
\]
\end{enumerate}
\end{corollary}

\begin{lemma}
\label{lemma:onepiece}
Let $S$ be coherent zonotopal subdivision of $A$ and let $B\subseteq A$ be a spanning subset. Then there is at most one $X\subseteq A\backslash B$, such that $[X,X\cup B] \in S$.
\end{lemma}
\begin{proof}
Let $w\in (\R^*)^n$ such that $S = S(Z(A),w)$. Since $B$ is of maximal dimension, there is at most one $\tilde w$ such that $\tilde w|_{\ker(\pi)} = w|_{\ker(\pi)}$ and $w\cdot b = 0$ for every $b\in B$. If such $\tilde w$ exists then the only tile of the form $[X, X\cup B]$ that is in $S$ is the one where $X = \{x\in A \mid \tilde w\cdot x < 0\}$. If no such $\tilde w$ exists then there is no tile of that form in the subdivision. 
\end{proof}

In the following result and in the rest of this section we denote by $A_J$ the subset of $A$ labelled by $J$, for any $J\subset [n]$.

\begin{lemma}
\label{lemma:type1}
Fix $k\ge 1$ and a lifting vector $w\in (\R^n)^*$, for a point configuration $A$ of size $n$.
For each tile $[X,Y] \subset  2^{[n]}$ 
such that $Y\backslash X$ a basis of $A$, 
the following are equivalent:
\begin{enumerate}

\item $\level{[X,Y]}{k+1}$ is a cell in $S^{(k+1)}(A,w)$.

\item There is an $x \in X$ such that 
$\level{[X\backslash x, Y\backslash x]}k$  is a cell in $S^{(k)}(A_{[n]\backslash x},w)$ 
but not in $S^{(k)}(A,w)$.

\item For every $x \in X$,
$\level{[X\backslash x, Y\backslash x]}k$  is a cell in $S^{(k)}(A_{[n]\backslash x},w)$ 
but not in $S^{(k)}(A,w)$.

\end{enumerate}
If, moreover, $k>1$, then they are also equivalent to:
\begin{enumerate}
\item[(4)] There are $x_1,x_2\in X$ such that 
$\level{[X\backslash x_i, Y\backslash x_i]}k$ is a cell in $S^{(k)}(A_{[n]\backslash x_i},w)$ for $i=1,2$.

\item[(5)] For every $x \in X$,
$\level{[X\backslash x, Y\backslash x]}k$  is a cell in $S^{(k)}(A_{[n]\backslash x},w)$.
\end{enumerate}
\end{lemma}

\begin{proof}
The implication (3)$\Rightarrow$(2) is obvious. 

To show that (2)$\Rightarrow$(1), consider an $x$ such that the cell $\level{[X\backslash x, Y\backslash x]}k$ is a cell in $S((A_{[n]\backslash x})^{(k)},w)$. 
Then by \Cref{prop:intersecting}, $[X\backslash x, Y\backslash x]$ is a cell of $S(Z(A_{[n]\backslash x}),w)$.
Therefore either $[X\backslash x, Y\backslash x]\in S(Z(A),w)$ or $[X,Y]\in S(Z(A),w)$ but not both by \Cref{lemma:onepiece}.
In other words, either $\level{[X\backslash x, Y\backslash x]}k\in S(A^{(k)},w)$ or $\level{[X,Y]}{k+1}\in S(A^{(k+1)},w)$ but not both. 
Since we assumed $\level{[X\backslash x, Y\backslash x]}k\not\in S(A^{(k)},w)$, we are done.

To see that (1)$\Rightarrow$(3), notice that if $\level{[X,Y]}{k+1}\in S(A^{(k+1)},w)$ then $[X,Y]\in S(Z(A),w)$. So for all $x \in X$ we have that the tile $[X\backslash x, Y\backslash x]$ is a cell of $S(Z(A_{[n]\backslash x}),w)$ and in particular $\level{[X\backslash x, Y\backslash x]}{k}\in S((A_{[n]\backslash h})^{(k)},w)$. But as $[X,Y]\in S(Z(A),w)$ then by \Cref{lemma:onepiece} $[X\backslash x, Y\backslash x]$ can not be a cell of $S(Z(A),w)$,
so $\level{[X\backslash x, Y\backslash x]}k$ can not be a cell of $S(A^{(k)},w)$.

Now assume that $k>1$. It is clear that (3)$\Rightarrow$(5)$\Rightarrow$(4). To see that (4)$\Rightarrow$(2) notice that it if $\level{[X\backslash x_i, Y\backslash x_i]}k\in S(A^{(k)},w)$ holds for $i =1,2$, then the two zonotopes $[X\backslash x_1, Y\backslash x_1]$ and $[X\backslash x_2, Y\backslash x_2]$ are in $S(Z(A),w)$, which can not happen by \Cref{lemma:onepiece}.
\end{proof}

\begin{proposition}
\label{prop:summand}
For every configuration $A$ of size $n$ and every $k\in[n-1]$ we have that
$\fiber^{(k+1)}(A)$ is a Minkowski summand of
\[
\fiber^{(k)}(A) +  \sum_{i \in[n]} \fiber^{(k)}(A_{[n]\backslash i}).
\]
\end{proposition}

\begin{proof}
Saying that $\fiber^{(k+1)}(A)$ is a Minkowski summand of $\fiber^{(k)}(A) +  \sum_{i \in[n]} \fiber^{(k)}(A_{[n]\backslash i})$ is equivalent to saying that if, for a given $w$ we know the subdivisions that $w$ induces in $\level Ak$ and in $\level{A\backslash x}{k}$ for every $x$ then we also know the subdivision induced in $\level A{k+1}$. 
For a cell $\level{[X,Y]}{k+1}$ with $|X| =k$, \Cref{lemma:type1} says that its presence in $S(\level A{k+1},w)$ is determined by its presence in $S(\level Ak,w)$ and $S(\level {A\backslash x}k,w)$. Cells $\level{[X,Y]}{k+1}$ with $|X|<k$ are in $S(\level A{k+1},w)$ if and only if $\level{[X,Y]}k\in S(\level Ak,w)$.
\end{proof}

The converse is only true for small $k$:

\begin{proposition}
\label{prop:equiv1}
For every configuration $A\subseteq \R^d$ of size $n$ and every $k\in[d]$ we have that
\[
\fiber^{(k+1)}(A)
\cong
\fiber^{(k)}(A) +  \sum_{i \in[n]} \fiber^{(k)}(A_{[n]\backslash i}).
\]
\end{proposition}
\begin{proof}
One direction is \Cref{prop:summand}. For the other direction we have that by \Cref{lemma:type1} then $S(A^{(k+1)},w)$ determines $S(A_{[n]\backslash i}^{(k+1)},w)$ for all $i\in[n]$. Any maximal cell in $\level{[X,Y]}k\in S(A^{(k)},w)$ must satisfy $|Y\backslash X|\ge d+1$, in particular $|Y|\ge d+1> k$, so $\level{[X,Y]}{k+1}$ is also a cell in $S(A^{(k+1)},w)$. This implies that $S(A^{(k+1)},w)$ determines $S(\level{A}{k},w)$.
\end{proof}

\begin{proposition}
For every configuration $A\subseteq \R^d$ of size $n>d+2$ and every $k\in[d]$ we have that
\[
\fiber^{(k)}(A) \le \sum_{i=1}^n \fiber^{(k)}(A_{[n]\backslash i})
\]
\end{proposition}
\begin{proof}
We need to prove that for every $w\in \R^d$, knowing $S(\level{A_{[n]\backslash i}}k ,w)$ for every $i$ determines $S(\level Ak,w)$. It is enough to prove it for a generic $w$, so we can assume the subdivisions are fine. Let $[X,Y]$ be a tile such that $Y\backslash X$ is an affine basis. We claim that $\level{[X,Y]}k\in S(\level{A_{[n]\backslash i}}k,w)$ if and only if $\level{[X\backslash i,Y\backslash i]}k\in S(\level{A_{[n]\backslash i}}k,w)$ for every $i\in [n]\backslash(Y\backslash X)$.

There is exactly one $\tilde w$ that agrees with $w$ in $\ker(\pi)$ and such that $\tilde w \cdot x = 0$ for every $x\in Y\backslash X$. 
We have that $\level{[X,Y]}k\in S(\level{A_{[n]\backslash i}}k,w)$ if and only if $\tilde w \cdot x<0$ for every $x \in X$ and $\tilde w \cdot x>0$ for every $x \in [n]\backslash Y$.
Notice that as $n> d+2$, $|[n]\backslash(Y\backslash X)|> 2$. Let $i \in [n]\backslash(Y\backslash X)$. 
As $k\le d$ and $|Y\backslash X|= d+1$, then for $Y\backslash i> k$ so $\level{[X\backslash i,Y\backslash i]}k$ is a full dimensional cell in the level $k$. 
So it is in $S(\level{A_{[n]\backslash i}}k,w)$ if and only if $\tilde w\cdot x< 0$ for every $x\in X\backslash i$ for all $x \in X\backslash i$ and $\tilde w \cdot x>0$ for every $x \in [n]\backslash (Y\cup i)$. As $|[n]\backslash(Y\backslash X)|> 2$, we can do this for two different elements in $[n]\backslash(Y\backslash X)$ so we can verify the sign of $\tilde w \cdot i$ for every $i\in [n]\backslash(Y\backslash X)$.
\end{proof}

A consequence of this is that \Cref{prop:equiv1} can be strengthened as follows:

\begin{proposition}
\label{prop:equiv2}
For every configuration $A\subseteq \R^d$ of size $n>d+2$ and every $k\in[d]$ we have that
\[
\fiber^{(k+1)}(A)
\cong
\sum_{i \in[n]} \fiber^{(k)}(A_{[n]\backslash i}).
\]
\end{proposition}
Notice that if $n = d+1$ then the fiber polytopes are just points and if $n=d+2$ they are just segments and in particular $\fiber^{(k+1)}(A) \cong \fiber^{(k)}(A)$.
Now we are ready to prove the main result of this section:
\begin{theorem}
\label{thm:fiber}
Let $A\subseteq \R^d$ be a configuration of size $n$ and $k\in [d+1]$. Let $s = \max (n-k+1, d+2)$. Then
\[
\fiber^{(k)}(A) \cong 
\sum\limits_{J\in {[n]\choose s}} \fiber(A_J)
\]
\end{theorem}
\begin{proof}
We prove this by iterating \Cref{prop:equiv2} several times. At each iteration, for $1<i \le k$, we replace each $\level\fiber {i+1}(A_J)$ by $\sum\limits_{j \in[n]} \fiber^{(i)}(A_{J\backslash j})$ if $|J| > d+2$ or by $\level\fiber i(A_J)$ if $|J| = d+2$. The iteration stops at level 1 with the desired result (notice that Minkowski sum is idempotent with respect to normal equivalence). 
\end{proof}

\begin{example}
\label{ex:hypersecondary}
Consider the regular hexagon $\polygon 6$. The secondary polytope $\level\fiber 1 (\polygon 6)$ is the 3-dimensional associahedron, as seen in \Cref{fig:associahedron}. Its border consists of 6 pentagons and 3 squares. By \Cref{thm:fiber}, the hypersecondary polytope $\level\fiber 2 (\polygon 6)$ is normally equivalent to the Minkowski sum of those 6 pentagons, see \Cref{fig:hyperassociahedron}. It has 66 vertices and the facets consist of 27 quadrilaterals (18 rectangles, 6 rhombi and 3 squares), 6 pentagons, 2 hexagons and 6 decagons. The short edges correspond to flips which do not change the set of vertices of the triangulation and the long edges correspond to those flips that do change the set of vertices. 

\begin{figure}[h]
	\centering
		\includegraphics[width = 0.6\textwidth]{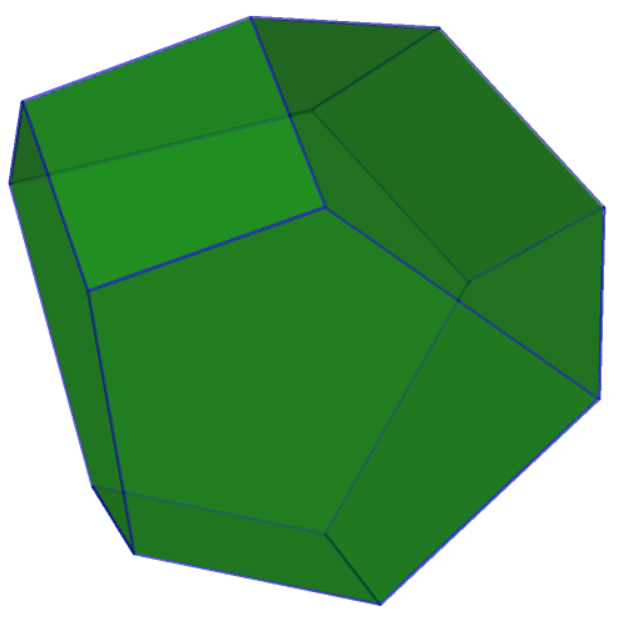}
	\caption{The associahedron $\level\fiber 1(\polygon 6)$.}
	\label{fig:associahedron}
\end{figure}

\begin{figure}[h]
	\centering
		\includegraphics[width = 0.6\textwidth]{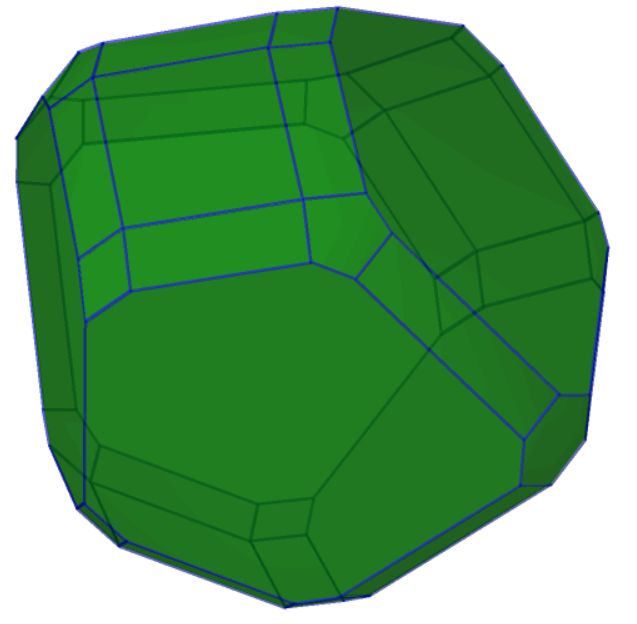}
	\caption{The hyperassociahedron $\level\fiber 2(\polygon 6)$.}
	\label{fig:hyperassociahedron}
\end{figure}
\end{example}

The GKZ vector corresponding to the triangulation from \Cref{ex:non_coherent} is in the center of one of the hexagons. There are 4 non-coherent hypertriangulations of $\level {\polygon 6}2$, which come in pairs with the same GKZ-vector, each in the center of one of the two hexagons. If instead of a regular hexagon we had a hexagon where the three long diagonals do not intersect in the same point, two of those subdivisions would become coherent and the hypersecondary polytope would have instead of each hexagon a triple of rhombi around the new vertex.

The order complex of the Baues poset $\level\baues 2(\polygon 6)$ is the (barycentric subdivision of the border of the) hyperassociahedron $\level\fiber 2(\polygon 6)$ where the hexagons are replaced by cubes. In particular it satisfies the Baues problem, that is, $\level\baues 2(\polygon 6)$ retracts onto $\level\fiber 2(\polygon 6)$. We will generalize this in \Cref{sec:baues}.

\section{Separation and lifting subdivisions}
\label{sec:separation}

Throughout this section let $A\subset\R^d$ be a point configuration labelled by $[n]$, and let  $Z(A)\subset\R^{d+1}$ be the zonotope generated by the vector configuration $A\times\{1\}\subset\R^d\times\{1\}$. Recall that a \emph{point} in $Z(A)$ is a subset $X\subset [n]$ and a \emph{tile} is an interval $[X,Y]\subset 2^{[n]}$, where $X\subset Y\subset [n]$.

Following~\cite{GaPo}, we say that two points $X_1, X_2 \subset [n]$ are \emph{separated with respect to $A$} or \emph{$A$-separated} for short if there is an affine  functional positive on $A_{X_1\backslash X_2}$ and negative on $A_{X_2\backslash X_1}$. Equivalently, if there is no oriented circuit $(C^+,C^-)$ in $A$ with $C^+\subset X_1\backslash X_2$ and $C^- \subset X_2\backslash X_1$.
Their motivation is that the notions of \emph{strongly separated} and \emph{chord separated} that were introduced in \cite{LeZe} and \cite{Galashin, OhPosSpe}
are equivalent to ``$\cyclic{n}{1}$-separated'' and ``$\cyclic{n}{2}$-separated'' respectively (\cite[Lemmas 3.7 and 3.10]{GaPo}).%
\footnote{Observe that \cite{OhPosSpe} uses the expression ``weakly separated'' for ``chord separated'', but ``weakly separated'' had a different meaning in \cite{LeZe}}
One of their main results is as follows (their statement is a bit more general, since it is stated for arbitrary oriented matroids, rather then ``point configurations''):

\begin{theorem}[\protect{\cite[Theorems 2.7 and 7.2]{GaPo}}]
\label{thm:maximal}
Let $A$ be a point configuration and let $m$ be the number of affinely independent subsets of $A$. Then:
\begin{enumerate}
\item No family of $A$-separated points in $A$ has size larger than $m$.
\item The map sending each zonotopal tiling to its set of vertices gives a bijection 
\[
\{\text{fine zonotopal tilings of $Z(A)$}\} \leftrightarrow
\{S\subset 2^{[n]} : S \text{ is $A$-separated and $|S|=m$}\}.
\]
\end{enumerate}
\end{theorem}

We here extend their definition to separation of tiles. In the rest of the paper we omit $A$ and write ``separated'' instead of $A$-separated:

\begin{definition}
Let $[X_1,Y_1]$ and $[X_2,Y_2]$ be two tiles. We say they are separated if there is no circuit $(C^+,C^-)$ such that $C^+\subset Y_1\setminus X_2$, $C^-\subset Y_2\setminus X_1$ and $C^+\cup C^- \not\subseteq (Y_1\cap Y_2)\setminus (X_1\cup X_2)$.
\end{definition}

The following diagram illustrates the circuits forbidden by the first two conditions in this definition. The third condition forbids circuits with support fully contained in the middle cell:
\[
\begin{array}{|c|c|c|c|}
\cline{2-4}
\multicolumn{1}{c|}{}&\quad X_2\quad & Y_2\setminus X_2 & [n]\setminus Y_2 \\
\hline
X_1 & 0 & \ge0 & \ge0 \\
\hline
Y_1\setminus X_1 & \le0 & * & \ge0 \\
\hline
[n]\setminus Y_1 &  \le0 &  \le0 & 0 \\
\hline
\end{array}
\]

By the orthogonality between circuits and covectors in an oriented matroid \cite[Proposition 3.7.12]{OMbook}, and the fact that covectors of a realized oriented matroid are the sign vectors of affine functionals this definition is equivalent to:

\begin{proposition}
\label{prop:separated}
Two tiles $[X_1,Y_1]$ and $[X_2,Y_2]$ are separated if there is a covector (that is, an affine functional) that is positive on $(X_1\setminus X_2) \cup (Y_1\setminus Y_2)$, negative on  $(X_2\setminus X_1) \cup (Y_2\setminus Y_1)$,  and zero on $(Y_1\cap Y_2)\setminus \{X_1\cup X_2\}$.
\end{proposition}

The following diagram illustrates the sign-patterns of covectors witnessing that two tiles are separated:
\[
\begin{array}{|c|c|c|c|}
\cline{2-4}
\multicolumn{1}{c|}{}&\quad X_2\quad & Y_2\setminus X_2 & [n]\setminus Y_2 \\
\hline
X_1 & * & + & + \\
\hline
Y_1\setminus X_1 & - & 0 & + \\
\hline
[n]\setminus Y_1 &  - &  - & * \\
\hline
\end{array}
\]

\begin{proof}
Consider the subset $I=(Y_1\cup Y_2) \setminus (X_1\cap X_2)$ of $A$, and let $A'$ be the restriction of $A$ to $I$. Remember that the circuits of $A'$ are the circuits of $A$ with support contained in $A'$, while the covectors of $A'$ are the covectors of $A$ (all of them) restricted to $A'$. In particular, the characterization of covectors of $A'$ as the sign vectors orthogonal to all circuits says that 
\[
\left((X_1\setminus X_2) \cup (Y_1\setminus Y_2) \, , \,
(X_2\setminus X_1) \cup (Y_2\setminus Y_1) \right)
\]
is a covector in $A'$ if and only if a circuit as in the definition of separation does not exist.
\end{proof}

\begin{example}
Two ``singleton tiles'' (that is, $X_1=Y_1$ and $X_2= Y_2$) are separated as tiles if and only if they are separated as points in the sense of Galashin and Postnikov.
Two tiles containing the origin, that is with $X_1=X_2=\emptyset$, are separated  if and only if $Y_1$ and $Y_2$ \emph{intersect properly} in the usual sense, as \emph{cells} in $A$.
Finally,  the whole zonotope $2^{[n]} = [\emptyset, [n]]$ is separated from a tile $[X,Y]$  if and only if the cells $Y$ and $[n]\setminus X$ intersect properly; this is 
equivalent to $[X,Y]$ being a face of the zonotope $Z(A)$.
\end{example}

The following result clarifies the relation between separation of points and tiles.
In it, we say that a tile $[X,Y]$ is \emph{fine} if $Y\backslash X$ is an independent set. Fine tiles are the ones that can be used in fine zonotopal tilings of $Z(A)$.

\begin{proposition}
\label{prop:fine-separated}
Let  $[X_1,Y_1]$ and $[X_2,Y_2]$ be two tiles.
If every point $B_1 \in[X_1,Y_1]$ is separated from every point $B_2 \in [X_2,Y_2]$, then $[X_1,Y_1]$ and $[X_2,Y_2]$ are separated. The converse holds if the tiles are fine. 
\end{proposition}

\begin{proof}
For the first direction, by induction on $|Y_1\backslash X_1| + |Y_2\backslash X_2|$, we can assume that $[X_1,Y_1]$ is not a singleton and that
every tile properly contained in it is separated from $[X_2,Y_2]$. In particular, taking any element $i\in Y_1\backslash X_1$ we have that 
both $[X_1\cup i,Y_1]$ and  $[X_1,Y_1\backslash i]$ are separated from $[X_2,Y_2]$.
By \Cref{prop:separated}, that implies the following two covectors:
\[
\begin{array}{|c|c|c|c|}
\cline{2-4}
\multicolumn{1}{c|}{}&\quad X_2\quad & Y_2\backslash X_2 & [n]\backslash Y_2 \\
\hline
X_1 & * & + & + \\
\hline
i & * & + & + \\
\hline
Y_1\backslash X_1 \backslash i& - & 0 & + \\
\hline
[n]\setminus Y_1 &  - &  - & * \\
\hline
\end{array}
\qquad
\begin{array}{|c|c|c|c|}
\cline{2-4}
\multicolumn{1}{c|}{}&\quad X_2\quad & Y_2\backslash X_2 & [n]\backslash Y_2 \\
\hline
X_1 & * & + & + \\
\hline
Y_1\backslash X_1 \backslash i & - & 0 & + \\
\hline
i & - & - & * \\
\hline
[n]\setminus Y_1 &  - &  - & * \\
\hline
\end{array}
\]
If $i\in X_2$ or $i\in [n]\backslash Y_2$ then the first or the second covector, respectively, show that $[X_1,Y_1]$ and $[X_2,Y_2]$ are separated. If $i\in Y_2\backslash X_2$ then elimination of $i$ in these two covectors gives a covector with values
\[
\begin{array}{|c|c|c|c|}
\cline{2-4}
\multicolumn{1}{c|}{}&\quad X_2\quad & Y_2\backslash X_2 & [n]\backslash Y_2 \\
\hline
X_1 & * & + & + \\
\hline
i &   & 0 &  \\
\hline
Y_1\backslash X_1 \backslash i& - & 0 & + \\
\hline
[n]\setminus Y_1 &  - &  - & * \\
\hline
\end{array},
\]
which again shows that $[X_1,Y_1]$ and $[X_2,Y_2]$ are separated.

For the converse, 
suppose first that $[X_1,Y_1]$ and $[X_2,Y_2]$ are separated and let $V$ be the covector showing it. Let $B_1$ and $B_2$ be points in them. Since the set $C:=(Y_1\setminus X_1) \cap (Y_2 \setminus X_1)$ is independent and is contained in the zero-set of $V$, no matter what signs we prescribe for its elements there is a covector $V'$ that agrees with $V$ where $V$ is not zero and has the prescribed signs on $C$. This implies the points $B_1$ and $B_2$ are separated.
\end{proof}

\begin{theorem}
\label{thm:separated}
Let  $[X_1,Y_1]$ and $[X_2,Y_2]$ be two tiles.
Then, the following conditions are equivalent:
\begin{enumerate}
\item The tiles are separated.
\item There is a zonotopal tiling of $Z(A)$ using both.
\item There is a coherent zonotopal tiling of $Z(A)$ using both.
\item There is a polyhedral subdivision of $A$ using $Y_1\backslash X_2$ and $Y_2\backslash X_1$ as cells.
\item There is a coherent polyhedral subdivision of $A$ using $Y_1\backslash X_2$ and $Y_2\backslash X_1$ as cells.
\end{enumerate}
\end{theorem}

\begin{proof}
Throughout the proof, let $A=\{a_1,\dots,a_n\}$ and denote $\tilde a_i =(a_i,1)$ the corresponding generator of $Z(A)$.

\begin{itemize}
	\item $1\Rightarrow 3$. Suppose the tiles are separated. By \Cref{prop:separated} this implies there is a linear functional $v\in (\R^{d+1})^*$ such that $v\cdot \tilde a_i $ takes the following values on the generators of $Z(A)$:
\[
	\begin{array}{|c|c|c|c|}
	\cline{2-4}
	\multicolumn{1}{c|}{}&\quad X_2\quad & Y_2\setminus X_2 & [n]\setminus Y_2 \\
	\hline
	X_1 & * & >0 & >0 \\
	\hline
	Y_1\setminus X_1 & <0 & 0 & >0 \\
	\hline
	[n]\setminus Y_1 &  <0 &  <0 & * \\
	\hline
	\end{array}
\]
Let $w\in (\R^n)^*$ be defined as follows on each $i\in [n]$:
\[
	\begin{array}{|c|c|c|c|}
	\cline{2-4}
	\multicolumn{1}{c|}{}&\quad X_2\quad & Y_2\setminus X_2 & [n]\setminus Y_2 \\
	\hline
	X_1 & -N & -2\, v\cdot \tilde a_i & -v\cdot \tilde a_i \\
	\hline
	Y_1\setminus X_1 & 0 & 0 & 0 \\
	\hline
	[n]\setminus Y_1 &  -v\cdot \tilde a_i &  -2\,v\cdot \tilde a_i & +N \\
	\hline
	\end{array}
\]
where $N$ is a very large positive number. Since $w$ is negative in $X_1$, positive in $[n]\backslash Y_1$, and zero in $Y_1\backslash X_1$, 
the tile selected by $w$ in the subdivision $S(Z(A),w)$ is $[X_1,Y_1]$. Similarly, the vector ${w'}\in(\R^n)^*$  defined by ${w'}_i = w_i +2 v\cdot \tilde a_i$ has the following values
\[
	\begin{array}{|c|c|c|c|}
	\cline{2-4}
	\multicolumn{1}{c|}{}&\quad X_2\quad & Y_2\setminus X_2 & [n]\setminus Y_2 \\
	\hline
	X_1 & <0 & 0 & v\cdot \tilde a_i \\
	\hline
	Y_1\setminus X_1 & 2\,v\cdot \tilde a_i & 0 & 2\,v\cdot \tilde a_i \\
	\hline
	[n]\setminus Y_1 &  v\cdot \tilde a_i &  0 & >0 \\
	\hline
	\end{array},
\]
which shows that $[X_2,Y_2]$ is also in $S(Z(A),w)$, since the difference between $w$ and $w'$ is a linear function.

\item $2\Rightarrow 1$. By the Bohne-Dress Theorem, zonotopal tilings of $Z(A)$ correspond to \emph{lifts} of the oriented matroid of $Z(A)$. Here, a lift is an oriented matroid $\M$ of rank $d+2$ on the ground set $[n+1]$ and such that $\M/(n+1) =\M(A)$. The tiles of the subdivision defined by the lift $\M$ are the intervals $[X,Y]\subset 2^{[n]}$ such that $\M$ has a covector that is negative on $X$, zero on $Y\setminus X$, and positive on $[n+1] \backslash Y$. 

That is, our hypothesis is that there is a lift $\M$ of $A$ that contains the covectors
\[
([n+1]  \setminus Y_1 \, ,\,  X_1)
\qquad\text{ and }\qquad
(X_2\, ,\, [n+1] \setminus Y_2).
\]
Elimination of the element $n+1$ among these covectors gives us a covector of \Cref{prop:separated}.

\item $1\Rightarrow 5$. Let $v$ as in the proof of $1\Rightarrow 3$, and define $w\in (\R^n)^*$ as follows:
\[
	\begin{array}{|c|c|c|c|}
	\cline{2-4}
	\multicolumn{1}{c|}{}&\quad X_2\quad & Y_2\setminus X_2 & [n]\setminus Y_2 \\
	\hline
	X_1 & N & 0 & 0 \\
	\hline
	Y_1\setminus X_1 & -v\cdot \tilde a_i & 0 & 0 \\
	\hline
	[n]\setminus Y_1 &  -v\cdot \tilde a_i &  -v\cdot \tilde a_i & N \\
	\hline
	\end{array}.
\]
Then $w$ and the $w'$ defined by $w'_i = w_i + v\cdot \tilde a_i$ show that $Y_1\backslash X_2$ and $Y_2\backslash X_1$ are cells in $S(A,w)$.

\item $4\Rightarrow 1$ For $C_1:=Y_1\backslash X_2$ and $C_2:=Y_2\backslash X_1$ to be cells in a subdivision it is necessary that their convex hulls intersect in a common face. 
That is, there must be a covector in $A$ that is zero in $C_1\cap C_2$, negative on $C_1\setminus C_2$, and positive on $C_2\setminus C_1$. 
These are precisely the same conditions as required in \Cref{prop:separated}.
\item $3\Rightarrow 2$ and $5\Rightarrow 4$ are obvious.
\qedhere
\end{itemize}
\end{proof}

\begin{remark}
With this theorem, it is now easy to see that \Cref{lemma:onepiece} also holds for non coherent subdivisions. If $Y_1\backslash X_1 = Y_2\backslash X_2$ is a spanning set then there can not be a linear functional vanishing on it, so $[X_1,Y_1]$ and $[X_2,Y_2]$ are not separated (unless $X_1 = X_2$, in which case they are the same cell).
\end{remark}

\begin{remark}
The definition of separated points and tiles makes sense for an arbitrary oriented matroid $\M$, since it uses only the notion of circuits, and \Cref{prop:separated} still holds in tis more general setting.

The notions of zonotopal tiling and of subdivision also make sense for arbitrary oriented matroids:
the former is interpreted as ``extension of the dual oriented matroid'' via \Cref{thm:BohneDress}
and the latter is studied in detail in \cite{OMtri}. In this setting the implications 
$(2) \Rightarrow (4) \Rightarrow (1)$ of \Cref{thm:separated} still hold, the first one as a consequence of 
the oriented matroid analogue of  \Cref{prop:lifting} and the second one because our proof above works at the level of oriented matroids. Yet:
\begin{enumerate}
\item The notion of \emph{coherent} subdivisions needs a realization of the oriented matroid be given. Not only the notion does not make sense for nonrealizable oriented matroids. Also, different realizations of the same oriented matroid may have different sets of coherent subdivisions, and non-isomorphic secondary polytopes/zonotopes.

\item 
The implication $(4) \Rightarrow (2)$ fails in the example of \cite[Section 5.2]{OMtri} (see Proposition 5.6(i) in that section), and the implication $(1) \Rightarrow (4)$ fails in the \emph{Lawrence polytope} that one can construct from that example.
\end{enumerate}

\end{remark}

\begin{corollary}
\label{prop:subtiles}
Let $[X_1,Y_1]$ and $[X_2,Y_2]$ be two separated tiles. Then any pair of subtiles $[\tilde{X_1},\tilde{Y_1}]\subseteq [X_1,Y_1]$ and $[\tilde{X_2},\tilde{Y_2}]\subseteq [X_2,Y_2]$ are separated.
\end{corollary}
\begin{proof}
By \Cref{thm:separated}, there is a zonotopal tiling using $[X_1,Y_1]$ and $[X_2,Y_2]$ and such tiling uses $[\tilde{X_1},\tilde{Y_1}]$ and $[\tilde{X_2},\tilde{Y_2}]$.
\end{proof}

\begin{proposition}
\label{prop:dependent_tiles}
Let $A$ be a configuration of $n$ pairwise independent points. Let $k\in [n-1]$.
Let $[X_1,Y_1]$ and $[X_2,Y_2]$ be two tiles that cover level $k$ (that is, $|X_i| < k < |Y_i|)$. Suppose that $[X_1,Y_1]$ and $[X_2,Y_2]$ are not-separated and that one of them is not fine.

Then, there are fine tiles $[X'_1,Y'_1]$ and $[X'_2,Y'_2]$  contained in $[X_1,Y_1]$ and $[X_2,Y_2]$, still covering level $k$ and still not separated.
\end{proposition}

\begin{proof}
By induction on the dependence rank of the tiles we only need to show that if $[X_1,Y_1]$ is dependent then there is a tile $[X'_1,Y'_1]$ properly contained in $[X_1,Y_1]$, covering level $k$, and non-separated from $[X_2,Y_2]$.

Let $(C_+,C_-)$ be a circuit showing that $[X_1,Y_1]$ and $[X_2,Y_2]$ are not-separated. Let $C=C_+\cup C_-$ be its support.

If there is an element $a\in (Y_1\backslash X_1) \backslash C$ then both $[X_1\cup a, Y_1]$ and $[X_1, Y_1\backslash a]$ are not separated from  $[X_2,Y_2]$, and one of them still covers level $k$, since dependent sets are of size at least 3. 

If there is no such an $a$, then $ Y_1\backslash X_1 \subset C$. 
Since $C$ is a circuit we conclude that $ Y_1\backslash X_1 = C$. 
By definition, we have that $C_-  \subset Y_2$ and $C_+ \subset [n]\backslash X_2$. 
Again, we take as new tile $[X_1\cup a, Y_1]$ or $[X_1, Y_1\backslash b]$, depending on which of the two still covers level $k$, where $a \in C_+$ and $b\in C_-$.
\end{proof}

\begin{corollary}
\label{coro:fine_is_enough}
Let $A$ be a point configuration in general position (``uniform'') and let $k\in [n-1]$. If no hypertriangulation of $A^{(k)}$ contains two non-separated tiles, then no hypersimplicial subdivision of $A^{(k)}$ contains them either.
\end{corollary}

\begin{proof}
Suppose that a subdivision $S$ contains two non-separated tiles  $[X_1,Y_1]$ and $[X_2,Y_2]$. Let $[X'_1,Y'_1]$ and $[X'_2,Y'_2]$  be the tiles guaranteed by 
\Cref{prop:dependent_tiles}. Then, we can refine $[X_1,Y_1]$ and $[X_2,Y_2]$ to fine subdivisions using $[X'_1,Y'_1]$ and $[X'_2,Y'_2]$. By general position this extends to a hypertriangulation refining $S$ and with two non-separated tiles.
\end{proof}

\section{Non-separated subdivisions}
\label{sec:non-separated}
We call a subdivision $S$ of $\level{A}{k}$ non-separated if it contains two non-separated cells. Non-separated subdivisions are certainly non-lifting.

\begin{example}
We here construct a non-separated subdivision in dimension two, which contrasts the fact that for $\polygon{n}$ such things do not exist~\cite{BalWel}.
Let $A$  be the configuration of the following 5 points in the plane: $p_1 = (1,2)$, $p_2 = (0,4)$, $p_3 = (4,4)$, $p_4 = (4,0)$ and $p_5 = (0,0)$. \Cref{fig:not_separated} on the right shows a hypertriangulation of $\level A2$ consisting of the triangles: 
\begin{gather*}
\level{[\emptyset,234]}2, \level{[\emptyset,245]}2,\\
\level{[2,1235]}2,\level{[2,2345]}2,\level{[4,1234]}2,\level{[4,1245]}2,\level{[4,1345]}2,\level{[5,1245]}2.
\end{gather*}
The circuit $(14,35)$ shows that the cell $\level{[2,2345]}2$ is not separated from the cells $\level{[4,1234]}2$ and $\level{[4,1245]}2$.

\begin{figure}[h]
\begin{minipage}{0.30\textwidth}
  \centering
  \includegraphics[width=0.80\textwidth]{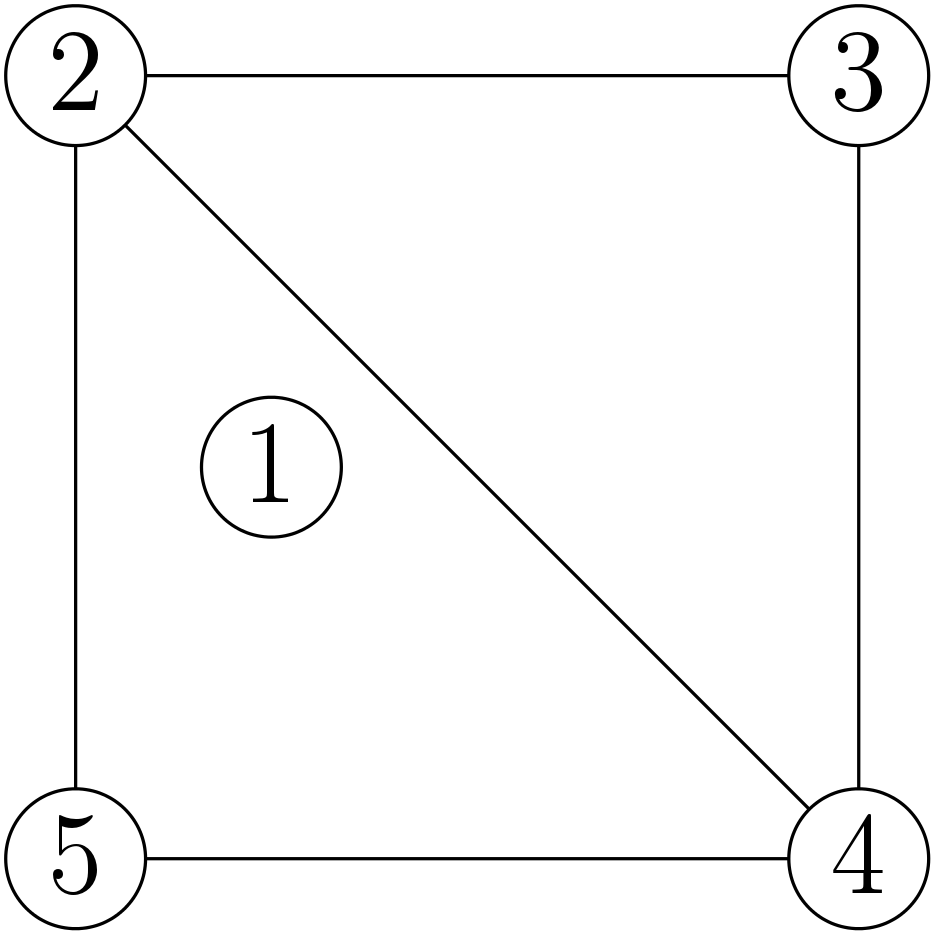}
\end{minipage}
\begin{minipage}{0.63\textwidth}
  \centering
  \includegraphics[width=0.80\textwidth]{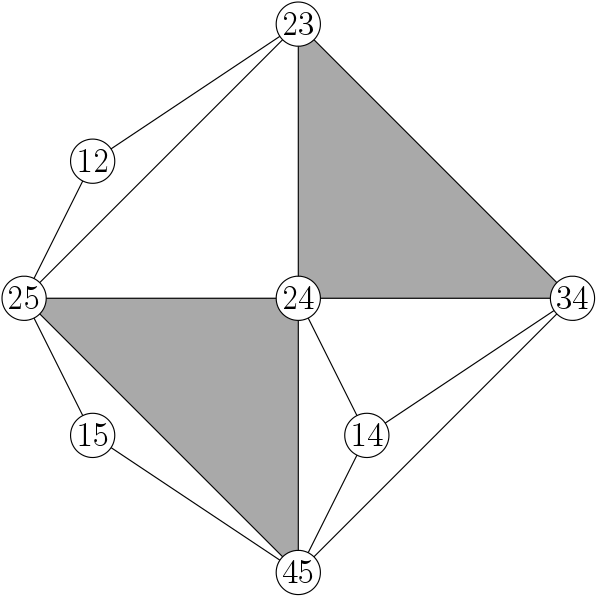}
\end{minipage}
\caption{A not separated hypertriangulation in the plane.}
\label{fig:not_separated}
\end{figure}

\end{example}

The following non-separated subdivision of $\level{\cyclic{4}{1}}2$ appears in \cite[Exm.~10.4]{Postnikov}:
\[
S=\{ 
\level{[1,123]}2,
\level{[1,134]}2,
\level{[4,124]}2,
\level{[4,234]}2
\}.
\]

Here we generalize it to

\begin{lemma}
\label{lem:non-separated}
For every odd $d$ and every $k\in [2,d-2]$ there is a non-separated hypertriangulation of $\level{\cyclic{d+3}{d}}k$.
\end{lemma}

\begin{proof}
A hypertriangulation of a configuration $A$ with $n=d+3$ has all its full-dimensional cells of one of the following forms, where $a < b\in [n]$ and we omit the superscript $(k)$, which will be clear from the context:
\[
{[\,\emptyset, \ [n]\backslash ab\,]}, \quad
{[\,a, \ [n]\backslash b\,]},\quad
{[\,b, \ [n]\backslash a\,]},\quad
{[\,ab, \ [n]\,]}.
\]
To simplify notation, we denote these four cells simply as $ab$, 
$\overline a b$, $a \overline b$ and $\overline a \overline b$, respectively (observe  that we always write the indices $a$ and $b$ in increasing order). For example, in this notation the subdivision $S$ of $\level{\cyclic{4}{1}}2$ mentioned above becomes
\[
S=\{ 
\overline14, \overline12, 3\overline4, 1\overline4
\}^{(2)}.
\]
One reason for this notation is that via the correspondence in \Cref{prop:lifting} the tile $[X,Y]$ corresponds in $\Gale{A}$ to the cone spanned by $X\cup\overline{[n]\setminus Y}$, where we use $\overline{B}$ to denote the set of vectors opposite to $B$, for $B\subset [n]$.

With this notation, \Cref{prop:lifting}(2) gives us that the following is a 
(coherent) zonotopal tiling of $Z(\cyclic{d+3}{d})$ (\Cref{fig:gale6} shows the case of $\cyclic63$):
\[
S_0 := 
\{\overline ab:  a \text{ odd}, b \text{ odd}\}\cup
\{\overline a\overline b:  a \text{ odd}, b \text{ even}\}\cup
\{ab:  a \text{ even}, b \text{ odd}\}\cup
\{a \overline b:  a \text{ even}, b \text{ even}\}.
\]

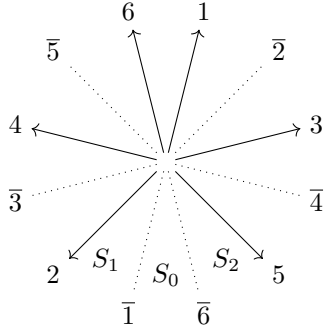
\begin{figure}[htb]
\begin{tikzpicture}[scale =.5]
  \node (O) at (0,0){}; 
  \node (P1) at (1,4) {1};
  \node (N2) at (3,3) {$\overline2$};
  \node (P3) at (4,1) {3};
  \node (N4) at (4,-1) {$\overline4$};
  \node (P5) at (3,-3) {5};
  \node (N6) at (1,-4) {$\overline6$};
  \node (N1) at (-1,-4) {$\overline1$};
  \node (P2) at (-3,-3) {2};
  \node (N3) at (-4,-1) {$\overline3$};
  \node (P4) at (-4,1) {4};
  \node (N5) at (-3,3) {$\overline5$};
  \node (P6) at (-1,4) {6};
  \draw[black,->] (O) -- (P1);
  \draw[dotted] (O) -- (N2);
  \draw[black,->] (O) -- (P3);
  \draw[dotted] (O) -- (N4);
  \draw[black,->] (O) -- (P5);
  \draw[dotted] (O) -- (N6);
  \draw[dotted] (O) -- (N1);
  \draw[black,->] (O) -- (P2);
  \draw[dotted] (O) -- (N3);
  \draw[black,->] (O) -- (P4);
  \draw[dotted] (O) -- (N5);
  \draw[black,->] (O) -- (P6);
  \node (S0) at (0,-3) {$S_0$};
  \node (S0) at (-1.6,-2.6) {$S_1$};
  \node (S0) at (1.6,-2.6) {$S_2$};
\end{tikzpicture}
\caption{The Gale transform of $\cyclic63$, with the regions corresponding to the zonotopal tilings $S_0$, $S_1$ and $S_2$ marked in it.}
\label{fig:gale6}
\end{figure}

$S_0$ admits the following cubical flips:
\begin{itemize}
\item Flip 1: negate the other symbol in every cell containing $\overline 1$. That is, remove
\[
\{\overline 1b:  b>1 \text{ odd}\}\cup
\{\overline 1\overline b:  b \text{ even}\}
\]
and insert
\[
\{\overline 1\overline b:  b>1 \text{ odd}\}\cup
\{\overline 1 b:  b \text{ even}\}.
\]

\item Flip 2: negate the other element in every cell containing $\overline n$. That is, remove
\[
\{a \overline n:  a<n \text{ even}\}\cup
\{\overline a\overline n:  a \text{ odd}\}
\]
and insert
\[
\{\overline a \overline n:  a<n \text{ even}\}\cup
\{ a\overline n:  a \text{ odd}\}.
\]
\end{itemize} 

These flips transform $S_0$ into two new coherent tilings $S_1$ and $S_2$, also shown in  \Cref{fig:gale6}. The two flips are not compatible, since both want to remove the tile $\overline1\overline n$ from $S_0$, and we can only remove it once. But $\overline1\overline n$  only affects level $1$ of the tiling, which means that in any  $\level{S_0}k$ with $k\ge 2$ we can do these two flips one after the other. After performing them we get a subdivision that contains (for $k\in [2,d-2]$) the non-separated cells
\[
\overline1 2 
\qquad \text{and}\qquad
n-1 \overline n.
\qedhere
\]
\end{proof}

To further generalize this construction we need the following easy lemma:

\begin{lemma}
\label{lemma:non-separated}
Let $A$ be a $d$-dimensional configuration of size $n$ in general position.
If $\level{A_{[n]\backslash i}}k$ has a non-separated subdivision $S$ for some $i\in[n]$
then $\level Ak$ and $\level A{k+1}$ have non-separated subdivisions too.
\end{lemma}

\begin{proof}
For $\level Ak$ do the following:
Extend $S$ to a subdivision $S'$ of $A$ by adding all the cells of the form $\level{[X, Y\cup i]}k$ with $[X,Y] \subset 2^{[n]}$ such that ${[X, Y\cup i]}$
is separated from $[\emptyset, [n]\backslash i]$. (The latter is equivalent to saying that $[X,Y]$ is contained in a facet of $Z(A_{[n]\backslash i})$ whose normal vector has positive scalar product with $i$). $S'$ is non-separated since it contains $S$.

For $\level A{k+1}$ apply the same construction upside-down. That is, consider the non-separated subdivision $\overline S$ of $\level{A_{[n]\backslash i}}{n-k-1}$ obtained from $S$ via the map $[X,Y] \to [[n]\backslash Y,[n]\backslash X]$. From $\overline S$ construct a non-separated subdivision $\overline{S'}$ of $\level{A}{n-k-1}$ as above, then turn $\overline{S'}$ upside-down  to get a non-separated subdivision of $\level{A}{k+1}$.
\end{proof}

\begin{corollary}
\label{coro:non-separated}
For every odd $d$, every $n\ge d+3$, and every $k\in [2,n-2]$,
there is a non-separated hypertriangulation of $\level{\cyclic{n}{d}}k$.
\qed
\end{corollary}

\begin{question}
Are there non separated hypertriangulations of $\level{\cyclic{n}{d}}k$ for $d\ge 4$ even? The case of $\cyclic{n}{2}$ suggests that the answer is no.
\end{question}

\section{Baues posets for $A=\polygon n$}
\label{sec:baues}

In this section we will restrict ourselves to the case when $A$ is a convex polygon $\polygon n$.
\begin{definition}
Let $S = \{\level{[X_i,Y_i]}{k_i}\}_{i\in I}$ be a subdivision of $\level{\polygon n}{k}$. We define 
\begin{align*}
S^+ &:= \{[X_i,Y_i] \mid i\in I \quad |Y_i| > k+1 \}\\
S^- &:= \{[X_i,Y_i] \mid i\in I \quad |X_i| < k-1 \}
\end{align*}
\end{definition}

\begin{proposition}
\label{prop:commute}
Let $S$ be a zonotopal tiling of $Z(\polygon n)$. Then
\[
\slevel{S}{k}{+} = \slevel{S}{k+1}{-}
\]
\end{proposition}
\begin{proof}
It is straightforward to check that both sets equal 
\[\left\{[X,Y] \in S \mid |X|<k \quad |Y|>k+1\right\}\]
\end{proof}

The proposition suggests we use the notation 
\[
\level{S}{k+\frac12} : = \slevel{S}{k+1}{-} = \slevel{S}{k}{+}
\]
and define the following poset:

\begin{definition}
We define $\level\baues{k+\frac{1}{2}}(\polygon n)$ to be the poset on the set
\[
\{\level{S}{k+\frac12} \mid S \in \baues^Z(\polygon n)\}
\]
where the order is refinement,  as in subdivisions: $S_1 < S_2$ if and only if $\forall \sigma \in S_1 \enspace \exists \tau \in S_2: \enspace \sigma \subseteq \tau$.
We have
 two natural order-preserving maps \ $\UP: \level\baues{k}(\polygon n) \to \level\baues{k+\frac{1}{2}}(\polygon n)$ and $\DOWN: \level\baues{k+1}(\polygon n) \to \level\baues{k+\frac{1}{2}}(\polygon n)$ such that for every $S \in \baues^Z(\polygon n)$ we have
\[
\UP(\level{S}{k}) = 
\DOWN(\level{S}{k+1})=
\level{S}{k+\frac12}.
\]
\end{definition}

\begin{remark}
The maps $\UP$ and $\DOWN$ are well defined thanks to the fact that all hypersimplicial subdivisions of $\polygon{n}$ are lifting (\cite{BalWel}). 
For more general configurations the definitions above would only make sense restricted to lifting subdivisions. 
\end{remark}

\begin{example}
\label{example:T}
\begin{figure}[h]
\begin{minipage}{0.54\textwidth}
  \centering
  \includegraphics[width=0.8\textwidth]{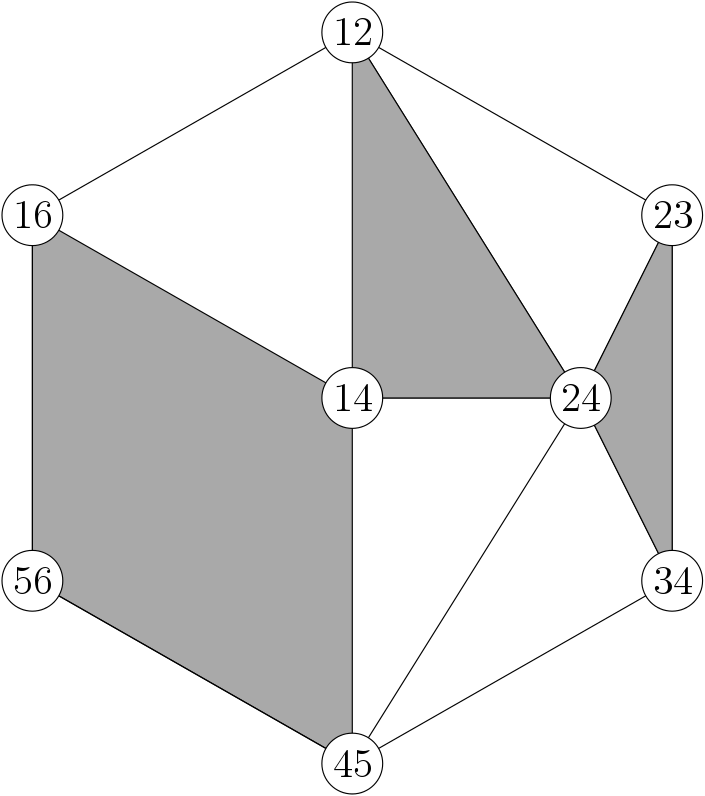}
\end{minipage}
\begin{minipage}{0.35\textwidth}
  \centering
  \includegraphics[width=0.8\textwidth]{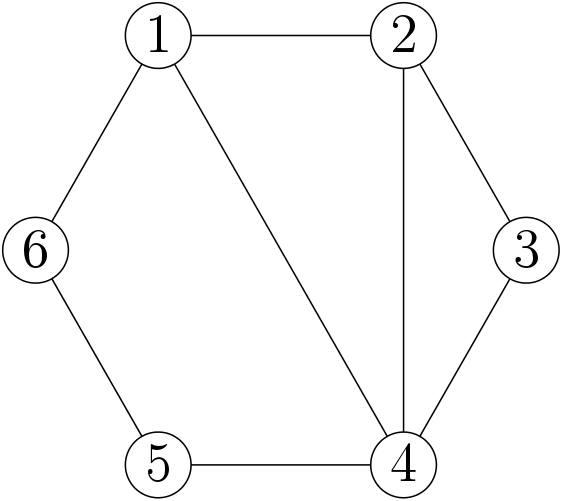}
\end{minipage}
\caption{The subdivision $T\in \level{\baues}{2}(\polygon{6})$ of \Cref{example:T} (left) and $\level{\DOWN(T)}{1}$ (right).}
\label{fig:T}
\end{figure}

Consider the subdivision $T\in \level{\baues}{2}(\polygon{6})$ in \Cref{fig:T} whose maximal cells are 
\[\left\{\level{[\emptyset, 124]}{2}, \level{[\emptyset,234]}{2},\level{[\emptyset,1456]}{2},\level{[1,1246]}{2},\level{[2,1234]}{2},\level{[4,1345]}{2},\level{[4,2345]}{2}\right\}\] 
The gray cells of $T$ in the figure give $\DOWN(T)$; that is:
\[\DOWN(T) = \{[\emptyset, 124],[\emptyset,234],[\emptyset,1456]\}.\] 
As seen in the right part of the figure, the cells in $\DOWN(T)$ are precisely the ones that have a full-dimensional intersection with the first level.
\end{example}

The main result in this section is that 
$\UP$ and $\DOWN$ induce homotopy equivalences of the corresponding order complexes (\Cref{coro:equivalence}).
To prove this we use the following criterion, originally proved by Babson~\cite{Babson1993}. Another proof can be found in~\cite{SturmfelsZiegler1993} and some generalizations appear in \cite{BWW2005}:

\begin{lemma}[Babson's Lemma]
\label{lemma:babson}
Let $f: \mathcal P \to \mathcal Q$ be an order preserving map between two posets. Suppose that for every $q\in \mathcal Q$ we have that
\begin{enumerate}
\item $f^{-1}(q)$  is contractible, and
\item $f^{-1}(q) \cap  \mathcal P_{\le p}$ is contractible, for every $p\in f^{-1}(\mathcal Q_{\ge q})$.
\end{enumerate}
Then $f$ is a homotopy equivalence.
\end{lemma}

For a collection $S$ of subzonotopes of $Z(A)$, let $\level{\vertices}{k}(S)$ be the set of vertices of cardinality $k$ of all zonotopes in $S$. We only consider a point $B$ in $[X,Y]$ to be a vertex if it is a face; that is, if $[X,Y]$ is separated from $\{B\}$.

\begin{proposition}
\label{prop:up}
Let  $S\in \level{\baues}{k+\frac{1}{2}}(\polygon n)$. Consider a point $X\in \level{\vertices}{k}(S)$. Define
\[
\up_S(X) := 
X \cup \{i\in [n] \mid X\cup i \in \vertices^{(k+1)}(S)\}.
\]
(Here ``\,$\up$'' stands for ``upper hole'').
Then $[X,\up_S(X)]$ is separated from every cell in $S$. 
\end{proposition}
\begin{proof}
Observe that $\up_S(X)$ equals
\[
\{i\in [n] \mid \exists j\in X \quad [X\setminus j, X\cup i] \text{ is a face of a cell in } S\}.
\]

Suppose there exists $X\in \level{\vertices}{k}(S)$ and $[I,J]\in S$ such that $[X, \cup\up_S(X)]$ and $[I,J]$ are not separated. 
Since $d=2$ we may assume that $|J\setminus I| \le 2$ and there is $Y\in \level{[X, \up_S(X)]}{k+2}$ such that $[X,Y]$ is not separated from $[I,J]$. 
So we have a circuit $(C^+, C^-)$ such that $C^+\in Y\setminus I$ and $C^-\in J\setminus X$
Further, since $S\in  \level{\baues}{k+\frac{1}{2}}(\polygon n)$ we can also assume $|I|\le k-1$. 
Let $y\in Y\setminus X$. Since $y\in \up_S(X)\setminus X$ we have that there is $x\in X$ such that $[X\setminus x, X\cup y]$ is a face of a cell in $S$. Then by \Cref{prop:subtiles} and the fact that $S$ is pairwise separated we have that $[X\setminus x, X\cup y ]$ is separated from $[I,J]$. So $C^+$ can not be contained in $X\cup y $. This means that $C^+ = Y\setminus X$. Notice that for every $i \in [n] \setminus C$ there is $y\in C^+$ such that $(C^+\setminus y \cup i, C^-)$ is a circuit. So if there is an $i \in X\setminus I$, this circuit would imply that $[X, Y\setminus y ]$ is not separated from $[I,J]$, which can not be as $[X, Y\setminus y]$ is a face of some cell in $S$. But this means $X\setminus I = \emptyset$ which is a contradiction since $|X|= k > k-1 = |I|$.
\end{proof}

\begin{corollary}
\label{coro:coarsest}
Let  $S\in \level{\baues}{k+\frac{1}{2}}(\polygon n)$. Then 
\[
\level{S}{k+1}\cup\{\level{[X, \up_S(X)]}{k+1} \mid X \in \level{\vertices}{k}(S) \},
\]
together with all their faces, form the unique coarsest subdivision in the fibre $\DOWN^{-1}(S)$.
\end{corollary}
\begin{proof}
We need to show that for $X_1, X_2\in \level{\vertices}{k}(S)$, $[X_1,\up_S(X_1)]$ and $[X_2,\up_S(X_2)]$ are separated. If not, we can again assume there are subsets $Y_1\subseteq \up_S(X_1)$ and $Y_2 \subseteq \up_S(X_2)$ of cardinality $k+2$ such that $[X_1, Y_1]$ and $[X_2,Y_2]$ are separated. As any subtile of them are faces of $S$, we have that there is a circuit $C^+ = Y_1\setminus X_1$ and $C^-= Y_2\setminus X_2$. Similarly as the proof of \ref{prop:up}, this implies that $X_1 = X_2$. The corollary follows from the fact that every cell in a subdivision in $\DOWN^{-1}(S)$ not coming from $S$ is of type 1 and hence it is contained in $[X,\up_S(X)]$ for some $X$.
\end{proof}

\begin{example}
\label{example:S}
Consider the subdivision $S\in \level{\baues}{1}(\polygon{6})$ in \Cref{fig:T} whose maximal cells are 
\[\left\{\level{[\emptyset, 124]}{1}, \level{[\emptyset,234]}{1},\level{[\emptyset,145]}{1},\level{[\emptyset,156]}{1}\right\}.\] 

We have that
\[
\begin{array}{ccc}
\up(1)= 12456, &\up(2)= 1234, &\up(3)= 234, \\
\up(4)= 12345, &\up(5)= 156, &\up(6)= 156, \\
\end{array}
\]
so that the coarsest subdivision $\hat S$ of $\DOWN^{-1}(S)$ has maximal cells 
\begin{align*}
\left\{\level{[\emptyset, 124]}{2}, \right.&\level{[\emptyset,234]}{2},\level{[\emptyset,145]}{2},\level{[\emptyset,156]}{2}, \\
&\left.
\level{[1, 12456]}{2}, \level{[2,1234]}{2},\level{[4,12345]}{2},\level{[5,1456]}{2}
\right\}.
\end{align*}
The two cells 
\begin{align*}
\level{[3, \up(3)]}{2} &= \level{[3, 234]}{2} \subset \level{[\emptyset,234]}{2}, \qquad\text{and}\\
\level{[6, \up(6)]}{2} &= \level{[6, 156]}{2}\subset \level{[\emptyset,1456]}{2} 
\end{align*}
 are also  in $\hat S$, but they are not maximal: they are edges.
\begin{figure}[h]

\begin{minipage}{0.35\textwidth}
  \centering
  \includegraphics[width=0.8\textwidth]{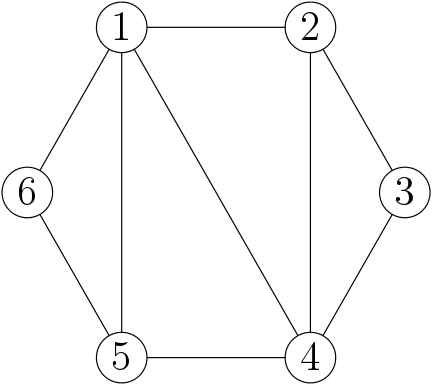}
\end{minipage}
\begin{minipage}{0.54\textwidth}
  \centering
  \includegraphics[width=0.8\textwidth]{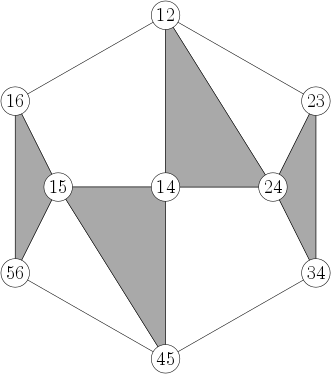}
\end{minipage}
\caption{The subdivision $S\in \level{\baues}{1}(\polygon{6})$ of \Cref{example:S} (left) and $\hat S \in \level{\baues}{2}(\polygon{6})$ (right).}
\label{fig:S}
\end{figure}
\end{example}

\begin{lemma}
\label{lemma:refinement}
Let  $S\in \level{\baues}{k+\frac{1}{2}}(\polygon n)$ and let $T\in \level{\baues}{k+1}(\polygon n)$ be such that $S \le \DOWN(T)$. 
Then, the poset $\DOWN^{-1}(S) \cap \level{\baues}{k+1}(\polygon n)_{\le T}$ has a unique maximal element.
\end{lemma}

\begin{proof}
Let $\hat S$ be the maximal element of $\DOWN^{-1}(S)$, as described in \Cref{coro:coarsest}. 

Let $T'\in \DOWN^{-1}(S)$, which is a refinement of $\hat S$. If a cell $\level{[X,Y]}{k+1}\in T'$ is such that $|X| < k$, then $[X,Y]\in S$ which implies that it is contained in a cell of $\DOWN(T)$. Then, $\level{[X,Y]}{k+1}$ is contained in a cell of $T$. Thus, for $T'$ to be a refinement of $T$, 
it is enough that $\level{[X,Y]}{k+1}\in T'$ is contained in a cell of $T$ for every $[X,Y]\in T'$ with  $|X|=k$.

For every such $X$, the cells $\level{[X,Y']}{k+1}\in T$ are a subdivision of the polygon $\level{[X,\up_{\DOWN(T)}(X)]}{k+1}$. Let $\level{[X,Y_1]}{k+1},\dots, \level{[X,Y_l]}{k+1}$ be such subdivision. For each $Y$ there are two possibilities:
\begin{itemize}
\item If $Y\subseteq \up_{\DOWN(T)}(X)$, then $\level{[X,Y]}{k+1}$ is contained in a cell of $T$ if and only if there is some $i\in[l]$ such that $Y \subseteq Y_i$. 
\item If $Y$ is not contained in $\up_{\DOWN(T)}(X)$, then $\level{[X,Y]}{k+1}$ is contained in a cell of $T$ if and only if $\level{[X,Y]}{k+1}$ does not intersect the interior of $\level{[X,\up_{\DOWN(T)}(X)]}{k+1}$. To see this, notice that if $\level{[X,Y]}{k+1}$ does not intersect the interior of $\level{[X,\up_{\DOWN(T)}(X)]}{k+1}$, then all vertices of $\level{[X,Y]}{k+1}$ correspond to edges of $\level{S}{k}$ contained in the same cell of $\DOWN(T)$. If this cell is $[X',Y']$, then $\level{[X',Y']}{k+1}\in T$ contains $[X,Y]$.
\end{itemize}
The discussion above implies that: a  $T'\in \DOWN^{-1}(S)$ is a refinement of $T$ if and only if all edges of $T$ are also edges in $T'$. This follows from the fact that the only edges in $T'$ not in $\hat S$ are of the form $[X,Y]$ with $|X|=k$ and $Y\subseteq \up_S(X)$. For each $X$, there is a unique coarsest subdivision of the polygon $\level{[X,\up_S(X)]}{k+1}$ that uses those edges. The subdivision that does that for each $X$ is the unique coarsest refinement of $T$ in $\DOWN^{-1}(S)$.
\end{proof}

\begin{example}
\label{example:tprime}
Consider the subdivisions $T$ from \Cref{example:T} and $S$ from \Cref{example:S}. We have that $S$ refines $\DOWN(T)$. The unique minimal, (actually, the only) subdivision in $\DOWN^{-1}(S) \cap \level{\baues}{k+1}(\polygon n)_{\le T}$ is $T'$ as depicted in \Cref{fig:tprime}.

\begin{figure}[h]
\includegraphics[width=0.4\textwidth]{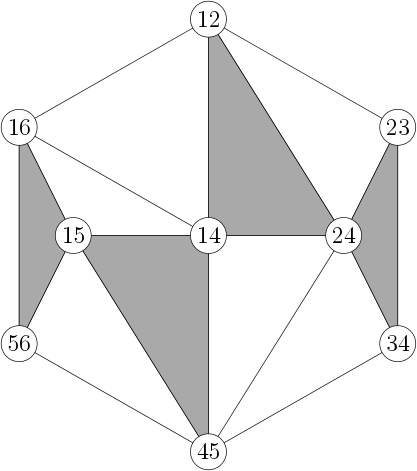}
\caption{The only subdivision $T'$ in $\DOWN^{-1}(S) \cap \level{\baues}{k+1}(\polygon 6)_{\le T}$}
\label{fig:tprime}
\end{figure}
\end{example}

\begin{remark}
\label{remark:coarsest}
One could expect the unique maximal element stated in \Cref{lemma:refinement} to coincide with the maximal element $\widehat{\DOWN(T)}$ in $\DOWN^{-1}(\DOWN(T))$. That is not the case in \Cref{example:tprime}. In fact, in that example $\widehat{\DOWN(T)}$ (whose picture would be as the picture of $T$ in \Cref{fig:T} without the edge $\{45,24\}$) does not refine $\hat S$.
\end{remark}

\begin{corollary}
\label{coro:equivalence}
The maps $\DOWN: \level{\baues}{k+1}(\polygon n) \to \level{\baues}{k+\frac{1}{2}}(\polygon n)$
and $\UP: \level{\baues}{k}(\polygon n) \to \level{\baues}{k+\frac{1}{2}}(\polygon n)$
are homotopy equivalences.
\end{corollary}

\begin{proof}
For $\DOWN$, conditions (1) and (2) in Babson's Lemma follow from \Cref{coro:coarsest} and \Cref{lemma:refinement}, respectively, since a poset with a unique maximal element is clearly contractible.
For $\UP$ the proof is completely symmetric.
\end{proof}

\begin{theorem}
\label{thm:equivalence}
Let $A$ be the vertex set of a convex $n$-gon.
The inclusion $\level{\baues_{\coh}}{k}(A) \to \level{\baues}{k}(A)$ is a homotopy equivalence, for $k=1,\dots, n-1$.
\end{theorem}

\begin{proof}
The proof is by induction on $k$. The base case, $k=1$, is the main result 
of Rambau and Santos in \cite{RaSa}.
Now let us suppose that $\level{\baues_{\coh}}{k}(A) \to \level{\baues}{k}(\polygon n)$ is a homotopy equivalence and we will prove that 
$\level{\baues_{\coh}}{k+1}(\polygon n) \to \level{\baues}{k+1}(\polygon n)$ is also a homotopy equivalence. 
Consider the following diagram, which commutes by \Cref{prop:commute}:

\begin{center}
\begin{tikzcd}
 & \level{\baues_{\coh}}{k+1}(\polygon n) \arrow[r, hook, dotted, "\level{i}{k+1}"] & [2em]
\level{\baues}{k+1}(\polygon n) \arrow[dr, "\DOWN"] & \\
\baues_{\coh}^Z(\polygon n) \arrow[ur,"\level{r}{k+1}"] \arrow[dr,"\level rk"] & & &
\level{\baues}{k+\frac{1}{2}}(\polygon n) \\
 & \level{\baues_{\coh}}{k}(\polygon n) \arrow[r, hook, "\level{i}{k}"] &
\level{\baues}{k}(\polygon n) \arrow[ur, "\UP"] & \\
\end{tikzcd}
\end{center}

The maps $\level{i}{k}$ and $\level{i}{k+1}$ are the inclusions of coherent subdivisions into all subdivisions. The maps $\level{r}{k}$ and $\level{r}{k+1}$ are the restriction of each zonotopal tiling to its $k$ and $k+1$ levels; that is, $S\mapsto \level{S}{k}$ and $S\mapsto \level{S}{k+1}$ respectively. They are homotopy equivalences since they can be geometrically realized as the identity maps among the normal fans of $\fiber^Z(\polygon n)$, $\level{\fiber}{k}(\polygon n)$ and $\level{\fiber}{k+1}(\polygon n)$. Since $\DOWN$ and $\UP$ are homotopy equivalences by \Cref{coro:equivalence}, and $\level{i}{k}$ is a homotopy equivalence by inductive hypothesis, the dotted arrow $\level{i}{k+1}$ must also be a homotopy equivalence.
\end{proof}

\begin{corollary}
The restriction map $\level{r}{k}: \baues^Z(\polygon n) \to \level{\baues}{k}(\polygon n)$ is a homotopy equivalence.
\end{corollary}

\begin{proof}
We now use the following commutative diagram:
\begin{center}
\begin{tikzcd}[column sep=large, row sep=large]
  \baues_{\coh}^Z(\polygon n) \arrow[r, hook, "\level{i}{k+1}"]
   \arrow[d, "\level{r}{k}"] & 
\baues^Z(\polygon n) \arrow[d, dotted, "\level{r}{k}"]  \\
  \level{\baues_{\coh}}{k}(\polygon n) \arrow[r, hook, "\level{i}{k}"] &
\level{\baues}{k}(\polygon n)   \\
\end{tikzcd}
\end{center}
The top arrow is a homotopy equivalence by \cite{SturmfelsZiegler1993} and the bottom arrow by \Cref{thm:equivalence}. The left arrow is also a homotopy equivalence, as mentioned in the proof of \Cref{thm:equivalence}, so the right arrow is a homotopy equivalence too.
\end{proof}

\section{Hypercatalan numbers}
\label{sec:hypercatalan}
Let $\level{C_n}k$ be the number of hypertriangulations of $\level{\polygon n}k$, which we will call hypercatalan number. 
When $k=1$ these are the usual Catalan numbers $C_n$. In this section we look at the case $k=2$. For a triangulation $T$ of $\polygon n$ and a vertex $i\in [n]$ we write $\deg_T(i)$ for the number of diagonals (edges excluding the sides of $\polygon n$) in $T$ incident to $i$ and we call it the \emph{degree} of $i$.

\begin{lemma}
\label{lemma:hypercatalan}
\[
\level{C_n}2 = \sum_T  \prod_{i\in [n]} C_{\deg_T(i)},
\]
where the sum runs over all trinangulations $T$ of $\polygon n$.
\end{lemma}
\begin{proof}
Let $T$ be a triangulation of $\polygon n$. To get a hypertriangulation of $\level{\polygon n}2$ that agrees with $T$ we need to triangulate $\level{[i,\UP_T(i)]}2$ for every $i$. As $\level{[i,\UP_T(i)]}2$ is a polygon with $\deg_T(i)+2$ vertices, the number of ways to triangulate it is $C_{\deg_T(i)}$. So for each triangulation $T$ there are $\prod\limits_{i\in [n]} C_{\deg_T(i)}$ hypertriangulations of $\level{\polygon n}{2}$. Summing over all triangulations gives the desired result.
\end{proof}

\begin{example}
For $n=3,\dots,10$  we have computed this formula to give the following values:
\[
\begin{array}{c|cccccccc}
n& 3 & 4 & 5 & 6 & 7 & 8 & 9 & 10\\
\hline
&1 &  2 &  10 &  70 &  574 &  5176 &  49656 &  497640 \\
\end{array}
\]

The computation for $n=6$ is as follows. Triangulations of the hexagon fall into three symmetry classes:
\begin{itemize}
\item Two triangulations with degree sequence $020202$, each contributing 
$1\cdot 2 \cdot1\cdot 2 \cdot1\cdot 2 =8$ to the sum.
\item Six triangulations with degree sequence $012012$, each contributing 
$1\cdot 1 \cdot2\cdot 1 \cdot1\cdot 2 =4$ to the sum.
\item Six triangulations with degree sequence $011103$, each contributing 
$1\cdot 1 \cdot1\cdot 1 \cdot1\cdot 5 =5$ to the sum.
\end{itemize}
This gives a total of 
$
2\cdot 8 + 6\cdot 4+6\cdot 5 = 70
$
fine subdivisions in $\level\baues{2}(\polygon 6)$.
\end{example}

\begin{lemma}
\label{lemma:catalan_product}
Let $T$ be a triangulation of an $(n+2)$-gon with $n \ge 4$. Then
\[
2^{n-2}   \le \prod_{i\in [n+2]} C_{\deg_T(i)}, \le  2^{\frac{5}{2}n -7}.
\]
\end{lemma}

\begin{proof}
Let $k_1\dots k_j$ be the sequence of the degrees of the vertices of $T$ which are positive.
The terms of this sequence add up to $2n-2$. The contribution of $T$ to the sum is $\prod_{i=1}^j C_{k_i}$. Observe that the number $(n+2)-j$ is the number of ears in $T$, which lies between $2$ and $\frac{n}2+1$. Thus, $j$ lies between $\frac{n}2+1$ and $n$.

For the lower bound, take into account that for every $k\ge 1$ one has $2^{k-1} \le C_k$,
we deduce the contribution of $T$ to be at least $2^{2n-2-j}$. Plugging in that $j\le n$, we get the desired lower bound.

For the upper bound, let $l$ be number of degree 1 vertices. Reorder the $k_i$ so that the last $l$ are equal to 1. We have that $\sum_{i=1}^{j-l} k_i = 2n-2-l$. Now take into account that for $k\ge 2 $ we have that $C_k\le 2^{2k-3}$, so
\[
\prod_{i\in [n+2]} C_{\deg_T(i)} = \prod_{i=1}^{j-l} C_{k_j} \le 2^{2(2n-2-l)-3(j-l)} = 2^{n-10+3e+l}
\]
where $e= n+2-j$ is the number of ears. So to prove the upper bound we need to show that $3e+l\le \frac{3n+6}{2}$. 

Suppose $T$ is the triangulation that maximizes $3e+l$. If there was a vertex of inner degree 1 such that it is not adjacent to an ear, flipping this edge would not decrease the number $3e+l$. So we can assume every degree 1 vertex is next to an ear. But then the vertex of degree 1 can not be neighbour to two ears, otherwise $n=2$, and it can not be neighbour to another vertex of degree 1, otherwise $n=3$. Also, an ear can not be neighbour to two degree 1 vertices, otherwise $n=2$. So the other neighbours of a pair of consecutive vertices (ear,degree 1) must have degree at least 2. Let $e'$ the number of ears not adjacent to any degree 1 vertex. Then $e-e' = l$ is the number of pairs (ear,degree 1) and we have:
\begin{align*}
l+2e &= 3l+2e' \le n+2\\
l+3e &\le n+2+e \le \dfrac{3(n+2)}{2}
\end{align*}
\end{proof}

\begin{corollary}
\label{coro:hypercatalan}
For $n\ge 6$,
\[
 2^{n-2}   \le \frac{\level{C_n}2}{C_n} \le  2^{\frac{5}{2}n -7}.
 \]
\qed
\end{corollary}

\begin{remark}
The lower bound of $2^{n-2}$ of \Cref{lemma:catalan_product} 
for the contribution of a single triangulation $T$
is attained by a zigzag triangulation, in which all degrees are $2$ except for two $1$s and two $0$s. 
When $T$ is a star triangulation in which a vertex is joined to all others, the contribution of $T$ is $C_{n-1} \sim 4^n$ (neglecting a polynomial factor). A higher contribution is obtained by the following procedure: start with any triangulation $T_0$ (e.g.~a zig-zag or a star). Let $T_1$ be obtained by adding an ear at each boundary edge of $T_0$, let $T_2$ be obtained from $T_1$ in the same way, etcetera. This method produces triangulations that contribute about $4.133^n$ (according to our computations) for $n$ large.
\end{remark}

\begin{remark}
By \cite[Theorem 1.2]{Galashin}, hypercatalan numbers are bounded from above by the number of fine zonotopal tilings of $Z(\polygon n)$, which is sequence A060595 in the Online Encyclopedia of Integer Sequences. The known terms are 
\[
\begin{array}{c|cccccccc}
n& 3 & 4 & 5 & 6 & 7 \\
\hline
&1 &  2 &  10 &  148 &  7686\\
\end{array}
\]
\end{remark}

\bibliography{hyperassociahedra}
\bibliographystyle{alpha}

\end{document}